\documentclass[12pt]{amsart}
\usepackage{amsfonts}
\usepackage{amsfonts,latexsym,rawfonts,amsmath,amssymb,amsthm}
\usepackage[plainpages=false]{hyperref}

\usepackage{graphicx}

\RequirePackage{color}

\numberwithin{equation}{section}

\newcommand{\beq}{\begin{equation}}
\newcommand{\eeq}{\end{equation}}
\newcommand{\beqs}{\begin{eqnarray*}}
\newcommand{\eeqs}{\end{eqnarray*}}
\newcommand{\beqn}{\begin{eqnarray}}
\newcommand{\eeqn}{\end{eqnarray}}
\newcommand{\beqa}{\begin{array}}
\newcommand{\eeqa}{\end{array}}

\newcommand{\bei}{\begin{itemize}}
\newcommand{\eei}{\end{itemize}}
 
\def\lan{\langle}
\def\ran{\rangle}
\def\p{\partial }
\def\eps{\epsilon}
\def\th{\theta}

\def\L{\mathcal L}
\def\M{\mathcal M}

\def\R{\Bbb R}

\def\Om{\Omega}
\def\pom{\p  \Omega}
\def\ol{\overline}
\def\noo{\noindent}

\def\bom{\overline\Omega}
\def\no{\nonumber}
\def\wtt{\widetilde}
\def\cR{\mathcal R}

\def\ou{\mathring u}
\def\U{\mathcal U}
\def\W{\mathcal W}

\newtheorem{Theorem}{Theorem}[section]
\newtheorem{Lemma}{Lemma}[section]
\newtheorem{Corollary}{Corollary}[section]
\newtheorem{Proposition}{Proposition}[section]

\title  {The higher dimensional affine Bernstein \\ and Plateau problems}

\begin{document}

\author {Neil S. Trudinger\ \ \ \ Xu-jia Wang}

\address{N.S. Trudinger: \ 
Mathematical Sciences Institute,
Australian National University, Canberra, ACT 0200, Australia.
}

\address{Xu-Jia Wang: \ 
Institute for Theoretical Sciences, 
Westlake University, Hangzhou, 310030, China.\\ 
}

\email{Neil.Trudinger@maths.anu.edu.au,\ \ \  Wangxujia@westlake.edu.cn}

\thanks {The first author was supported by the Australian Research Council (DP230100499). }



\bibliographystyle{plain}

\begin{abstract}
In our previous papers \cite{[TW1], [TW2]}, 
we established the affine Bernstein theorem
and proved the interior regularity for the affine Plateau problem in two dimensions.
In this paper, we prove the interior regularity for the affine Plateau problem
in all dimensions and present a smooth counterexample to the affine Bernstein problem in dimensions ten and higher.
We also prove an affine Bernstein theorem in all dimensions under a growth condition.

\vskip20pt

\centerline{\bf Notes added on October 7th, 2026}
\vskip10pt

\noindent
Today OpenAI released nearly 400 papers solving major conjectures in mathematics. One of them, titled 

\vskip5pt

\centerline{\it A Smooth Nonquadratic Entire Affine Maximal Graph in Dimension Ten}

\vskip5pt

\noindent
gives a counterexample to the affine Bernstein problem in dimension 10. This result is already included in Section 3 of this paper. The paper was submitted on February 2, 2026. The results of this paper have been presented at several conferences this year, including in March at Shanghai Jiao-Tong University and in May at the “National Workshop on Geometric Analysis 2026” at Zhuhai, China.

 \end{abstract}

\maketitle

\baselineskip=16.8pt
\parskip=5pt

\section{Introduction}

In this paper, we investigate the affine Bernstein and affine Plateau problems, 
proposed by Chern and Calabi \cite { [Ca1], [Ch]}.
These are the affine invariant analogues 
of the celebrated Bernstein and Plateau problems for minimal surfaces. 
The underlying affine metric was introduced in the early twentieth century,
see e.g. Blaschke's famous treatise \cite{[B]}.
Since then affine geometry has become a classical subject of differential geometry 
and continues to be a focus of extensive research.

The affine mean curvature equation is a nonlinear fourth order PDE of the form
\beq \label{a1.1}
 L[u]:={\Small\text{$\sum_{i, j=1}^n $}} U^{ij}w_{x_ix_j}=f(x),    \ \ \ x\in\R^n,       
 \eeq
where $(U^{ij})$ is the cofactor matrix of the Hessian
matrix $D^2 u=(u_{x_ix_j})$, which is assumed to be positive, and $w$ is given by
\beq\label{a1.2}
 w=  [\det D^2 u]^{\theta-1},\ \ \ \theta={\Small\text{$\frac{1}{n+2}$}}.
\eeq
When $f=0$, equation \eqref{a1.1} is the affine maximal surface equation.
One can also regard equations \eqref{a1.1} and \eqref{a1.2} as a system of two second order PDEs, 
where \eqref{a1.1} is a linearized Monge-Amp\`ere equation and \eqref{a1.2} is a classical Monge-Amp\`ere equation.
Therefore the study of the affine maximal surface equation relies heavily on techniques and results 
from the theory of Monge-Ampère type equations \cite {[C1], [C2], [C3], [CG], [Fi], [Gu], [Le25]}.

The Bernstein problem for minimal surfaces has been a core problem in geometric analysis
ever since Bernstein proved that an entire two-dimensional minimal graph is a plane.
This result was extended to dimensions $n \le 7$ 
by de Giorgi for $n=3$, Almgren for $n=4$, Simons for $n=5,6,7$, 
after which a counterexample, for all dimensions $n\ge 8$, 
was discovered by Bombieri, de Giorgi, and Giusti; see \cite{[BDG]}.
It still remains a subject of considerable research interest to date \cite{[Be], [CL24]}.
The Bernstein problem for fully nonlinear elliptic equations has also received much attention.
For the Monge-Amp\`ere equation, it was proved by Jörgens for $n = 2$, Calabi  for $2\le n\le 5$,
and Pogorelov for all $n\ge 2$, that an entire convex solution must be a quadratic polynomial \cite{[P]}. 
See also \cite{[W16a], [W16b], [Y02]} for the Bernstein problem for other fully nonlinear elliptic equations.

For the affine Bernstein problem,
Chern \cite{[Ch]} conjectured that an entire affine maximal graph 
must be that of a quadratic polynomial in dimension $n=2$.
Calabi \cite{[Ca1]} proved that an affine maximal surface in $\R^3$ 
is an elliptic paraboloid if it is both Euclidean complete and affine complete.
The Chern conjecture was finally resolved by the authors in  \cite{[TW1]},
which stimulated further researches on the topic \cite {[AMM1], [AMM2], [CSZ], [Du20], [Du24], [Le17],  [Zh]}.
A conjecture on higher dimensional affine Bernstein problem was proposed in  \cite{[TW5]}, which states
\begin{itemize}
\item []{\it
When $n\le 9$, a Euclidean complete, locally uniformly convex affine maximal hypersurface in $\R^{n+1}$ 
must be an elliptic paraboloid. When $n \ge 10$, 
there is a smooth, locally convex affine maximal hypersurface in $\R^{n+1}$ which is not an elliptic paraboloid.}
 \end{itemize}

In this paper, we provide a smooth counterexample to the affine Bernstein problem
in dimensions $n\ge 10$ by reducing equation \eqref{a1.1} to a fourth order ODE (Theorem \ref{Thm2.1}), 
thereby resolving the second part of the above conjecture.
In \cite{[TW1]}, a non-quadratic entire affine maximal graph was constructed for dimension $n=10$, 
but the function has a singularity at the origin. 
Its graph is an affine maximal cone, 
in the sense that its level sets are self-similar up to affine transforms.
This non-smooth function in \cite{[TW1]} is used in this paper  to prove 
that the ODE admits a global smooth solution and thus to construct a smooth counterexample
to the  affine Bernstein problem.
We also prove a Bernstein theorem in all dimensions under a growth condition 
(Theorem \ref{Thm2.2}) that excludes our counterexample. 

The affine Plateau problem can be formulated as a geometric variational problem.
The existence of maximizers follows from the upper semi-continuity and concavity
of the functional \cite{[TW2], [Lu]}.
As in the classical Plateau problem, the regularity is the main focus of interest.
In our previous paper \cite{[TW2]}, we were able to prove the interior regularity of maximizers in two dimensions. 
In this paper, we establish the interior regularity in all dimensions. 
The results are stated in Theorem \ref{Thm2.3} for the first boundary value problem for the affine maximal surface equation,
which corresponds to the graph case, and  in Theorem \ref{Thm2.4} for the general affine Plateau problem.

Note that in equations \eqref{a1.1} and \eqref{a1.2}, 
the parameter $\theta=\frac {1}{n+2}$ corresponds to  the affine invariant case.
When $\theta=0$, equation \eqref{a1.1} is known as Abreu’s equation.
For general $\theta\ge 0$,  \eqref{a1.1} represents an important class of fourth order PDEs
that has been extensively studied in recent years
\cite{[CW], [KLWZ], [Le13], [Le16], [Le20], [Le21],  [LZ21], [WZ23]}.
The affine invariance at $\theta=\frac{1}{n+2}$ is particularly useful 
for establishing interior regularity for the affine Plateau problem.

A special case of the affine Plateau problem is the first boundary value problem
for the affine maximal surface equation. This involves prescribing both the solution $u$ and its gradient
 $Du$ on the boundary, 
\beqn
  u & =& \phi\ \ \ \ \ \text{on}\ \ \pom, \label{a1.3} \\
 Du & =& D\phi\ \ \ \text{on}\ \ \pom,  \label{a1.4} 
\eeqn
where $\phi$ is a smooth, uniformly convex function. Note that when
the boundary value \eqref{a1.3} is given,  the gradient condition \eqref{a1.4} is
a single condition, equivalent to
prescribing the normal derivative of $u$ on the boundary.

The affine mean curvature equation \eqref{a1.1} is the Euler equation of
the functional
\beq \label{a1.5}
J(u)=A(u) - {\small\text{$\frac{1}{n+2}$}} \int_\Om f(x)\,udx,
\eeq
where
\beq\label{a1.6}
A(u)=\int_\Om \big[\det D^2 u\big]^{\frac{1}{n+2}}
\eeq
is the affine area functional. A function which maximizes the
functional $A$ under local convex perturbation is called {\it affine
maximal} and its graph is called an affine maximal hypersurface.
Accordingly its Euler equation is called the {\it affine maximal surface
equation}. The functional $A$ is the volume functional
of the affine metric
\beq \label{a1.7}
g =[\det D^2 u]^{-\frac{1}{n+2}}u_{x_ix_j}dx_idx_j.
\eeq
It is Riemannian, and accordingly equation \eqref{a1.1} is elliptic, when
the Hessian matrix $D^2 u$ is positive, namely when
$u$ is locally uniformly convex.

\vskip5pt

Following the strategy of \cite{[TW2]}, 
our approach to the interior regularity of the affine Plateau problem proceeds in two main stages. 
The regularity of {\it strictly convex} maximizers is first established
through a priori estimates and an approximation argument 
based on penalization by the second boundary value problem. 
This reduces the final stage of the proof to establishing strict convexity,
which necessitated the restriction to two dimensions in \cite{[TW2]}.

Recall that for the Monge-Ampère equation, 
the strict convexity of solutions to the Dirichlet problem 
was proved under an affine boundary condition  \cite{[C1]}. 
For the second boundary value problem, 
strict convexity was established under the assumption that the target domain is convex \cite{[C3]}. 
In contrast, for the fourth-order affine maximal surface equation, 
the issue of strict convexity, particularly in higher dimensions, 
is considerably more complicated, and the affine boundary data alone is not enough to guarantee the strict convexity. 
In this paper, we establish strict convexity using novel techniques. 
First, we improve the approximation in \cite{[TW2]} 
by a secondary penalization; 
this allows us to apply a priori estimates from the dual problem of the affine volume functional. 
Second, through a delicate convexity analysis, 
we demonstrate precisely how the regularity condition on the boundary 
ensures the solution to be strictly convex.
   
A further subtlety in the affine Plateau problem is that 
the key approximation (Theorem  \ref{Thm4.1}) cannot be applied directly, 
as it relies on the uniqueness of maximizers of \eqref{a1.9},
a property that may fail for problem  \eqref{a1.10}. 
This necessitates a novel localization of the a priori estimates, 
of which the details are presented in Section 6.

This paper is arranged as follows. 
In Section 2, we first state the main results of the paper,
introduce the dual problem, then recall the a priori estimates,
and establish an upper bound for the Gauss curvature of the maximizer, 
for the affine maximal surface equation. 
Section 3 addresses the affine Bernstein problem and contains the proofs of Theorems \ref{Thm2.1} and \ref{Thm2.2}.
The remainder of the paper is devoted to the affine Plateau problem.
In Section 4 we establish the improved approximation result. 
In Section 5 we prove the strict convexity of the maximizer of \eqref{a1.9}, 
thereby completing the proof of Theorem \ref{Thm2.3}. 
Finally, in Section 6 we study the general affine Plateau problem,
proving the strict convexity of maximizers of \eqref{a1.10} and thus establishing Theorem \ref{Thm2.4}.



\section{\bf  Main Results and Preliminaries}

\subsection{Main Results} 

The first result of the paper is a counterexample to the affine Bernstein problem.

\begin{Theorem} \label{Thm2.1}
When the dimension $n\ge 10$,
there exists an entire smooth convex solution to the affine maximal surface equation 
which is not a quadratic polynomial.
 \end{Theorem}
 
Theorem \eqref{Thm2.1} will be proved by considering a special class of functions and reduced 
the affine maximal surface equation to a fourth order ODE.

We then prove a Bernstein theorem under a growth condition.

\begin{Theorem} \label{Thm2.2}
Let $u$ be a smooth, convex solution to the affine maximal surface
equation.  Then $u$ is a quadratic polynomial
if there exists a positive constant $\delta>0$ such that either 
\beq \label{growa}
u(\lambda x)\ge (\lambda+\delta)u(x) \ \ \text{as}\ \ |x|\to\infty,
\eeq
for a fixed constant $\lambda>1$; or
\beq \label{growb}
x\cdot Du(x)\ge (1+\delta) u(x)\ \ \text{as}\ \ |x|\to\infty .
\eeq
 \end{Theorem}

In \cite{[EH90]}, Ecker and Huisken studied the Bernstein problem for minimal surfaces.
They proved that a global minimal surface must be a hyperplane in all dimensions
if it satisfies the condition
\beq\label{growc}
|Du(x)|=o(|x|+u(x))\ \ \text{as}\ |x|\to\infty.
\eeq
Our condition \eqref{growb} is inspired by \eqref{growc}. 
The above counterexample shows that Theorem \ref{Thm2.2} does not hold 
if $\delta=0$ in \eqref{growa} or \eqref{growb}.
 
Next we consider the first boundary value problem.
To formulate this problem as a variational maximization problem \cite{[Ca1]}, we 
denote
$$
\begin{array}{l}
\ol S[\phi, \Om]=\{v\in C^0(\bom)\ |\ v\ \text{is convex},
  v=\phi\ \text{on}\ \pom,  N_v(\Om) \subset \overline{N_\phi(\Om)}\},
 \end{array} $$
where $N_v$ is the normal mapping of $v$. 
If $\phi$ is Lipschitz continuous, $N_\phi(\Om)$ is bounded,
and by convexity of functions, the set $\ol S[\phi, \Om]$ is compact.
Hence a solution to the first boundary value problem \eqref{a1.1}-\eqref{a1.4} is necessarily a
maximizer of the functional $J$ in $\ol S[\phi, \Om]$.  In \cite{[TW2]} we
proved the existence of a unique maximizer to the variational
problem
\beq\label{a1.9}
\sup\{J(u)\ |\ u\in \ol S[\phi, \Om]\}, 
\eeq
together with  its interior regularity in dimension two. In this paper we prove

\begin{Theorem}\label{Thm2.3} 
Let $\Om$ be a smooth, bounded domain in $\R^n$
and $\phi$ a smooth, uniformly convex function in $\bom$.
Then
\bei
\item [(i)]  if $f\in L^\infty(\Om)$, 
the maximizer of \eqref{a1.9} is uniformly convex and belongs to the Sobolev space $W^{4,p}_{loc}(\Om)$
for any $p>1$; 

\item [(ii)] if $f\in C^\alpha(\Om)$ for some $\alpha\in
(0, 1)$, the maximizer is  $C^{4,\alpha}$ smooth
in $\Om$.
\eei
\end{Theorem} 

We say a function $u\in C(\Om)$ is {\it convex} if
$\forall\ x_0\in \Om$,  $u\ge \L_{x_0}$ in $\Om$ for some
affine function $L_{x_0}$ for which $L_{x_0}(x_0) = u(x_0)$;
$u$ is {\it locally convex} if $u\ge \L_{x_0}$ near the point $x_0$;  
$u$ is strictly convex if $u(x)>\L_{x_0}(x)$ $\forall\ x\in\Om, x\ne x_0$;
and $u$ is uniformly convex if $u\in C^2(\Om)$ is strictly convex and the Hessian matrix $D^2 u$ is positive definite. 
Similarly we can define convexity, strict convexity, and uniform convexity of locally
convex hypersurfaces in $\R^{n+1}$ \cite{[TW2]}.

The affine Plateau problem, as an extension of the first boundary value problem, 
can also be formulated as a variational problem \cite{[Ca1]}.
Let $\M_0$ be a bounded, connected, smooth, and
locally uniformly convex hypersurface in $\R^{n+1}$, with smooth
boundary $\Gamma =\p\M_0$. Similarly we introduce

\begin{itemize}


\item [{$\ol S[\M_0]$:}]  the set of locally convex hypersurfaces $\M$
with boundary $\Gamma $, such that $\M$ can be continuously deformed from $\M_0$ by a family
of locally convex hypersurfaces $\M_t$, such that $\M_{t=1}=\M$,
and the Gauss mapping image of $\M_t$ for $t\in (0, 1)$ is contained in 
that of $\M_0$.
\end{itemize}

\noo
When $\M_0$ is the graph of a convex function $\phi$, by the convexity, one easily verifies that the set $\ol S[\M_0]$
coincides with $\ol S[\phi, \Om]$.

The affine Plateau problem is to determine a hypersurface
$\M\in \ol S[\M_0]$ which maximizes the functional
\beq \label{a1.10}
A(\M)=\sup\{ A(\M')\ {|}\ \M'\in \ol S[\M_0]\},
\eeq
where $A(\M)=\int_\M K^{1/(n+2)}$ is the affine area functional 
and $K$ is the Gauss curvature. 
The functional \eqref{a1.10} coincides with \eqref{a1.6} when $\M$ is the graph of $u$.

In \cite{[TW2]} we proved the existence of maximizers to \eqref{a1.10} in all
dimensions and the interior regularity of maximizers in dimension
two. Here we prove the regularity of maximizers in all
dimensions.

\begin{Theorem} \label{Thm2.4} 
Let $\M_0$ be a bounded, connected, smooth,
and locally uniformly convex hypersurface in $\R^{n+1}$, with $C^\infty$ smooth
boundary $\Gamma =\p\M_0$. Suppose the Gauss image of $\M_0$ does not
contain any hemisphere. Then the maximizer of \eqref{a1.10} is strictly
convex and $C^\infty$ smooth.
\end{Theorem}

The assumption that the Gauss image of $\M_0$ does not contain any
hemisphere is necessary for the existence of maximizers of \eqref{a1.10} \cite {[TW2]},
but it is not needed for our proof of the regularity.

\subsection{Existence of maximizers} 
First we recall 
the existence and uniqueness of maximizers to the variational problem \eqref{a1.9}.
We extend the affine area functional $A$ to non-smooth convex
functions by setting
\beq \label{a2.1}
A(u)=\int_\Om \big[\det \p^2 u\big]^{\frac{1}{n+2}}, 
\eeq
where $\p^2 u(x)=D^2 u(x)$ when $u$ is twice differentiable at $x$ 
and $\p^2 u(x)=0$ otherwise. 
It was proven in \cite{[Lu], [TW1], [TW2]} that 
the extended functional is concave and upper semi-continuous with
respect to uniform convergence. The existence of a maximizer  $u$ follows
from  the upper semi-continuity. In \cite{[TW2]} we proved that the
singular part of the Monge-Amp\`ere measure $\mu[u]$ vanishes in
$\Om$. Hence the concavity of $A$ also implies if $u_1$ and $u_2$
are two maximizers, then the Monge-Amp\`ere measures
$\mu[u_1]=\mu[u_2]$. Accordingly we have the following existence and uniqueness 
of maximizers.

\begin{Lemma}  \cite{[TW2]} \label{L2.1} 
Let $\Om$ be a bounded open set in $\R^n$
with Lebesgue measure $|\pom|=0$,  $\phi$ a convex, uniformly Lipschitz function
defined in a neighborhood of $\bom$, and $f\in L^\infty(\Om)$. Then
there is a unique maximizer $u$ of \eqref{a1.9}.
\end{Lemma}

\noo {\bf Remark 2.1}.
\newline
(i) In Theorem 2.1, \cite{[TW2]},  we assume that $\pom$ is Lipschitz.
But actually the assumption $|\pom|=0$ suffices for the existence
and uniqueness of maximizers. 
Indeed, if $|\pom|=0$,  
the set $D_\delta:=\{x\in\Om\ |\ \text{dist}(x, \pom)<\delta\}$ has small
Lebesgue measure. But for any measurable set $E$,
we have
\beqn 
 \int_E [\det D^2 u]^{1/(n+2)}
  &=& \int_{\M_E} K^{1/(n+2)}  \label{a2.2}  \\
  &\le& |M_E|^{(n+1)/(n+2)} \int_{\M_\delta} K\nonumber\\
  &\le & \omega_n |M_E|^{(n+1)/(n+2)}\to 0\nonumber
  \eeqn
as $|E|\to 0$, where $\M_E$ is the graph of $u$ over $E$, $K$ is the
Gauss curvature, and $\omega_n$ is the surface area of the unit
sphere $S^n$. 

\noo (ii) 
In  case {\bf c1.2)} of the proof of Theorem \ref{Thm6.1},
the domain $\Om$  may not be smooth.
If either $\pom$ or $\phi$ is not smooth, we assume that $\phi$ is defined in a neighborhood of $\bom$, and satisfies
\beq\label{phib}
 {\bigcap}_{\delta>0} N_\phi(\Om^\delta)=\overline{N_\phi(\Om) } . 
 \eeq
In this paper we denote 
\beq  \label{a2.4}   
\begin{split}
 &\Om^\delta=\{x\in\R^n\ |\ \text{dist}(x, \Om)<\delta\},\\
 &\Om_\delta=\{x\in\Om\ |\ \text{dist}(x, \pom)>\delta\}.
\end{split} \eeq
If both $\pom$ and $\phi$ are smooth, we can smoothly extend $\phi$ to a neighborhood of $\bom$.
Moreover, if $\phi$ is locally uniformly convex, one can preserve the local uniform convexity
for the extended $\phi$ \cite{[TW4]}.

\vskip5pt

In \cite{[TW2]} we proved that a maximizer can locally be approximated
by smooth ones. Therefore by the a priori estimates (see \S2.4
below), a maximizer of \eqref{a1.9} is smooth provided it is strictly convex. 
Consequently, 
establishing strict convexity becomes the central issue to proving regularity.  
The following examples demonstrate that a maximizer is not, in general,  strictly convex.

\vskip5pt

\noo{\bf Example 1}. Let $\phi(x)= (1+x_n^2)|x'|^{2-\frac 2n}$ be
Pogorelov's function, where $n\ge 3$ and $x'=(x_1, \cdots,
x_{n-1})$. Then $f:=\det D^2 \phi\ge 1$ and $f$ is smooth near the
origin. Let $u$ be a solution of $\det D^2 u=1$ in $B_r(0)$ such
that $u=\phi$ on $\p B_r(0)$. Then $u$ is affine maximal but $u$ is
not strictly convex.

\vskip5pt

\noo{\bf Example 2}. Let $n=10$ and $u(x)=\sqrt{|x'|^9+x_{10}^2}$.
Then $u$ is affine maximal but $u$ is not strictly convex \cite{[TW1]}.


\vskip5pt

In this paper we prove that the maximizer of \eqref{a1.9} is strictly convex if
$\phi$ is uniformly convex. The above examples show that such a
condition is necessary. Unlike the Monge-Amp\`ere equation
\cite{[C1]}, vanishing boundary trace is not enough for the strict
convexity for affine maximal hypersurfaces, as shown by Example 2.

\subsection{The Legendre transform} 

The argument in this paper relies heavily on the dual variational problem 
which arises from the Legendre transform.
First let us recall some basic properties of the transform. 
Let $\Om$ and $\Om^*$ be two domains in $\R^n$. Let
$u$ be a convex function in $\Om$. 
The Legendre transform of $u$  is the dual function $u^*$ in $\Om^*$ given by
\beq   \label{a2.3}  
u^*(y) =\sup\{x\cdot y-u(x), \ \ \ x\in\Om\} \ \ \ \forall\ \ y\in\Om^*.
\eeq
Formally \eqref{a2.3} is well-defined for any domain $\Om^*\subset\R^n$. Obviously
\newline
$\text{\hskip10pt}$ (i) if $u\ge v$, then $u^*\le v^*$; therefore
\newline
$\text{\hskip10pt}$  (ii) if $u\ge 0$ and $u=0$ on the segment $\{(0, \cdots, 0, t),
-1<t<1\}$,  then $u^*\ge |y_n|$; 
\newline
$\text{\hskip10pt}$  (iii) if $u_m\to u$, then $u_m^*\to u^*$.

In this paper we will apply the Legendre transform to functions in $\ol S[\phi, \Om]$. 
For any $u\in \ol S[\phi, \Om]$, in order that $u^*$ is well defined in $N_\phi(\Om)$,
we assume that $\phi$ is a convex function defined in $\Om^\delta$ for a small $\delta>0$.
For any $u\in \ol S[\phi, \Om]$, we extend it to
$\Om^\delta$ such that $u=\phi$ in $\Om^\delta-\Om$. 
Then $u^*$ is defined in $N_\phi(\Om^\delta)$, and
$u^*(y)=\phi^*(y)$ for $y\in N_\phi(\Om^\delta)-\Om^*$. 

When $u$ is smooth and uniformly convex, and $\Om^*=Du(\Om)$, the
sup in \eqref{a2.3} is attained at $x\in\Om$, uniquely determined by
$y=Du(x)$, and
$$u^*(y)= x\cdot y-u(x).$$
Moreover, the dual function $u^*$ is smooth, uniformly convex in
$\Om^*$, and
\beq   \label{a2.5}  
(u^*)^*=u,
\eeq
namely the Legendre transform of $u^*$ is $u$ itself. By  property (iii)
above, \eqref{a2.5} is true for nonsmooth functions.

From the relation $y=Du(x)$ we have $x=D u^*(y)$ and
\beq   \label{a2.6}  
\{D^2 u(x)\}=\{D^2 u^*(y)\}^{-1}. 
\eeq
Therefore we have
\beq   \label{a2.7}  
J(u)=J^*(u^*),
\eeq
where $J$ is given in \eqref{a1.5} and
$$J^*(u^*)= A^*(u^*)
      -\frac {1}{n+2}\int_{\Om^*} f(Du^*)(yDu^*-u^*)\det D^2 u^* $$
is the dual functional, where
\beq   \label{a2.8}  
A^*(u^*)= \int_{\Om^*}\big[\det D^2 u^*\big]^{\frac {n+1}{n+2}}.
\eeq
If $u$ is a smooth, convex maximizer of the functional $J$, $u^*$ is
a maximizer of the associated functional $J^*$, and satisfies the
Euler equation of $J^*$,
\beq   \label{a2.9}  
{\begin{split}
L^*[u^*]:  =U^{*ij} w^*_{ij}
          & = - {\Small\text{$\frac {1}{n+2}$}} f(Du^*)\,\det D^2 u^*, \\ 
          \end{split}}
\eeq
where $U^{*ij}$ is the cofactor matrix of $D^2 u^*$ and 
$w^*=[\det D^2 u^*]^{-1/(n+2)}$. 
Equation \eqref{a2.9} is the dual of \eqref{a1.1}.
To verify \eqref{a2.9}, denote
$$E(u^*)=\int_{\Om^*} f(Du^*)(yDu^*-u^*)\det D^2 u^*.$$
We can compute the first variation of $E$ directly,
$$\frac {d}{dt} E(u^*+tw)_{\big| t=0}
           =-\int_\Om w f(Du^*)\det D^2 u^*.$$

If the maximizer $u$ is not smooth, by (iii) above and approximation
by smooth convex functions, we find that \eqref{a2.7} still holds \cite{[TW2]}.
Therefore $u\in\ol S[\phi, \Om]$ is a maximizer of $\sup_{\ol
S[\phi, \Om]} J(u)$ if and only if $u^*$ is a maximizer of
$\sup_{\ol S[\phi^*, \Om^*]} J^*(u)$. Accordingly $u^*$ is a
variational solution to the boundary value problem
\beqn   \label{a2.10}  
 L^*[v] &=& -{\Small\text{$\frac {1}{n+2}$}}  f(Du^*)\,\det D^2 u^*
           \ \ \ \ \ \ \text{in}\ \ \ \Om^*,    \\
       v  &=& \phi^* \ \ \ \ \text{on}\ \ \pom^*,  \nonumber \\
     Dv &=&  D\phi^*\ \ \text{on}\ \ \pom^*. \nonumber  
\eeqn

\subsection{Monge-Amp\`ere measure and strict convexity.} 

We will use the following result of Caffarelli \cite{[C1]}.

\begin{Lemma}  \label{L2.2}  
Let $u$ be a convex function in $\Om$,
vanishing on $\pom$. Suppose there exist positive constants $C_1,
C_2$ such that
\beq   \label{a2.11}  
C_1\le \det D^2 u\le C_2\ \ \ \text{in}\ \ \Om.
\eeq
Then $u$ is strictly convex and $C^1$ smooth, 
and for any subdomain $\Om'\subset\subset\Om$,  
the modulus of convexity of $u$ and 
the modulus of continuity of $Du$ in $\Om'$
depend only on $n, C_1, C_2, \Om',$ and $\Om$.
\end{Lemma}

Obviously the condition that $u$ vanishes on $\pom$ can be relaxed
to that $u$ is linear on $\pom$, as the Monge-Amp\`ere equation is
invariant after subtracting a linear function. 
Note that Lemma \ref{L2.2} applies to nonsmooth convex functions as well. 
In the nonsmooth case the Monge-Amp\`ere
operator $\det D^2 u$ is replaced by the {\it Monge-Amp\`ere
measure} $\mu[u]$, which is defined by
\beq   \label{a2.12}  
\mu[u](E)= |N_u(E)|, 
\eeq
where $|\cdot |$ denotes the Lebesgue measure and $N_u$ is the
normal mapping of $u$,
$$N_u(x_0)=\{y\ |\ \ u(x)\ge y\cdot x+u(x_0)\ \forall\ x\in\Om\}.$$
For any subset $\omega\subset\Om$, we denote $N_u(\omega)=
\bigcup_{x\in\omega}N_u(x)$.  
An important property of the Monge-Amp\`ere measure is the 
weak convergence, namely if a sequence of convex function
$u_m\to u$ pointwise in $\Om$, 
then $\mu[u_m]$ converges to $\mu[u]$ weakly as measures \cite{[P]}.
In the following we use $\det D^2 u$ to denote the measure
$\mu[u]$ even if $u$ is nonsmooth.

\vskip5pt

The modulus of convexity for a convex function $u$ is defined as
follows \cite{[TW2]}. For any point $y\in \Om$ and constant $h>0$, denote
\beq   \label{a2.13}  
S^0_{h, u}(y)=\{x\in\Om\ {|}\  u(x)<\L_y(x)+h\},
\eeq
where $\L_y$ is a tangent plane of $u$ at $y$. Then the {\it modulus
of convexity} of $u$ is a nonnegative function, given by
$$\rho_u (r) = \inf_{y\in\Om} \rho_{u, y}(r), \ \ \ \ r>0,$$
where
$$\rho_{u, y}(r) = \sup
 \{h\ge 0\ \big{|} \ S^0_{h,u}(y)\subset B_r(y)\}  $$
if there exists $h\ge 0$ such that $S^0_{h,u}(y)\subset B_r(y)$;
or $\rho_{u, y}(r) = 0$ otherwise.  A convex function $u$ is {\it
strictly convex} if $\rho_u(r)>0$ for all $r>0$.

\vskip5pt

\noo{\bf 2.4. A priori estimates}
\newline
We will use a basic estimate from \cite{[TW1], [TW2]}, 
for the more general equation
\beqn 
  && U^{ij}w_{ij}=f,   \label{a2.14}  \\
  && w=[\det D^2 u]^{\th-1},  \no 
 \eeqn
where $(U^{ij})$ is the cofactor matrix of the Hessian matrix $D^2
u$, and $\th\in (0, 1)$ is a constant. For later applications,
it suffices to consider two cases,
\\[5pt]
$\text{\hskip10pt}$ Case (i):  $\th=\frac {1}{n+2}$, $f\le C^*$ for some constant $C^*$;
\\
$\text{\hskip10pt}$ Case (ii): $\th=\frac {n+1}{n+2}$, $f=-{ \frac {1}{n+2} } \hat f\,\det
D^2 u$ and $\hat f\ge -C^*$.

\begin{Lemma}  \label{L2.3}  
Let $u\in C^4(\Om)\cap C^{0,1}(\bom)$ be a
locally uniformly convex solution of \eqref{a2.14} with $u=0$ on $\pom$.
Suppose $f$ satisfies either (i) or (ii) above. Then, for any point
$y\in\Om$, we have the estimate
\beq  \label{a2.15}  
\det D^2 u(y)\le C, 
\eeq
where $C$ depends on $n, C^*$, $\text{dist} (y, \pom)$, $\sup_\Om|Du|$, and $|\inf_\Om u|$.
\end{Lemma}

\noo{\it Proof.} Let
\beq   \label{a2.16}  
z=\log \frac w{(-u)^\beta} -A|Du|^2, 
\eeq
where $\beta= \frac{n(n+1)}{n+2}$, $A=\frac {1-\th}{4\th \sup_\Om
|Du|^2}$. Then $z$ attains its minimum at an interior point
$x_0\in\Om$. At $x_0$ we have $z_i=0$, namely
\beq  \label{a2.17}  
\frac {w_i}w =\beta \frac {u_i}u +2Au_ku_{ki},
\eeq
and $\{z_{ij}\}\ge 0$, where
$$z_{ij}=\frac {w_{ij}}w-\frac {w_iw_j}{w^2}-\beta\frac
 {u_{ij}}u +\beta\frac {u_iu_j}{u^2}-2Au_{ki}u_{kj}-2Au_ku_{kij}. $$
Hence
\beqs
 0 & \le & u^{ij}z_{ij} \\
& = & \frac {f}{d^\th} -\frac {\beta n}{u} -
    \frac {u^{ij}w_iw_j}{w^2} + \frac {\beta u^{ij} u_iu_j}{u^2}
    -2Au^{ij}u_{ki}u_{kj}+\frac {2A}{1-\th}\frac {u_k w_k}{w}\\
& = &\frac {f}{d^\th} -\frac {\beta n}{u}
    -\beta(\beta-1)\frac {u^{ij} u_iu_j}{u^2}
    -2A\Delta u +\frac {4A^2 \th}{1-\th} u_{ij}u_iu_j
    -2\beta A \frac {1-2\th}{1-\th} \frac {|Du|^2}{u} \\
& \le & \frac {f}{d^\th} -A\Delta u -\frac {\beta n}{u}
    +2\beta A \frac {|Du|^2}{|u|}. 
\eeqs
In the above we have used \eqref{a2.17} and the formula
$$u^{ij} u_{kij} = -\frac {1}{1-\th}\frac {w_k}w,$$
where $(u^{ij})=\frac 1d (U^{ij})$ is the inverse matrix of $D^2 u$
and $d=\det D^2 u$.

In Case (i), we obtain
$$ A |u| \Delta u \le C(1+|Du|^2)+\frac{C^*}{d^\th}. $$
Hence $z(x_0)\ge -C$ and so $z(x)\ge z(x_0)\ge -C$ for any
$x\in\Om$.

In Case (ii),  we have
$$A |u| \Delta u\le C(1+|Du|^2)+C^* d^{1/(n+2)}.$$
We also obtain \eqref{a2.15}. $\square$

The above proof was already contained in \cite{[TW2]}. 
Our new contribution is Case (ii) which will be needed later.
We include the full proof here to clarify how the constant C depends on the right hand
side function $f$. In particular the estimate \eqref{a2.15} is independent
of $\inf f$ in Case (i), and independent of $\sup\hat f$ in case
(ii).

By the same proof, we have a variant of Lemma \ref{L2.3} as follows.

\begin{Lemma}  \label{L2.3a} 
Let $u\in C^4(\Om)\cap C^{0,1}(\bom)$ be a
locally uniformly convex solution of \eqref{a2.14} in $\Om$.
Suppose that\\
\centerline{ $\pom=\Gamma_1\cup\Gamma_2$, \ 
$u=0$ on $\Gamma_1$, \ and $\det D^2 u\le C_0$ on $\Gamma_2$.}
Suppose $f$ satisfies either (i) or (ii) above. Then we have the estimate \eqref{a2.15}, 
where the constant $C$ depends on $n, C^*$, $C_0$, $\text{dist} (y, \pom)$, $\sup_\Om |Du|$,
and $\sup_\Om |u|$.
\end{Lemma}

\noo{\bf Remark 2.2}.  Lemma \ref{L2.3} will be used in the following scenario. 
Let $u$ be a convex function defined near the origin, satisfying 
\bei
\item [(a)] $u(0)=0$, and 
\item [(b)] $u(x)\ge C\max(|x'|^2, |x_n|)$ for some positive constant $C>0$. 
\eei
Conditions (a) and (b) imply that
the Monge-Amp\`ere measure $\mu_u$ cannot be a bounded function near
the origin. Indeed, under (a) and (b), one can easily verify that
$$|N_u(D_\eps)|/|D_\eps|\to\infty\ \ \text{as} \ \ \eps\to 0,$$
where $D_\eps=\{x\in\R^n\ |\ u(x)<\eps\}$. 
When $u$ is a maximizer of \eqref{a1.9},  by approximation 
it is in contradiction with the a priori estimate \eqref{a2.15}. 
Hence the maximizer $u$ cannot satisfy (a) and (b) simultaneously. 

\vskip5pt

From Lemma \ref{L2.3} and the affine invariance, 
we can also derive an upper bound for the Gauss curvature of affine maximal hypersurfaces.
Let $\M$ be an affine maximal hypersurface such that 
\bei
\item [(i)] the origin $0\in\M$, 
\item [(ii)] $\M$ is bounded and contained in the strip $\{0\le x_{n+1}\le 1\}$, and
\item [(iii)]  the boundary $\p \M\subset \{x_{n+1}=1\}$.
\eei
Let $Q_\M$ denote the convex domain enclosed  by $\M$ and the hyperplane $\{x_{n+1}=1\}$.
By making a linear transformation, 
we may assume that $Q_\M$ is {\it normalized} (i.e. its minimuim ellipsoid is a ball and (i) - (iii) hold).
Then we have the following upper bound for the Gauss curvature of $\M$.

\begin{Lemma} \label{L2.4}
Under the above normalization, there exists a constant $h_0\in (0, \frac 12)$ such that
the Gauss curvature of $\M$ at $p$, $K_\M (p)$, satisfies the estimate
\beq\label{KC}
K_\M (p)\le C\ \ \forall\ p=(p_1, \cdots, p_n, p_{n+1})\in\M\ \text{with}\ 0\le p_{n+1}\le h_0,
\eeq
where $h_0$ and $C$ depend only on $n$.
\end{Lemma}

\noo{\it Proof}. 
If $\M$ is uniformly convex, \eqref{KC} follows from \eqref{a2.15}, 
for a constant $C$ depending also on the modulus of convexity of $\M$.

To obtain an upper bound independent of the modulus of convexity of $\M$, we argue by contradiction.
If \eqref{KC} is not true, then there is a sequence $\M_k$ and a sequence of points $p_k\in \M_k\cap\{x_{n+1}<\frac 1k\}$ 
such that $K_{\M_k} (p_k)\ge k$.
Since $\M_k$ is convex and is normalized, we may assume that $\M_k\to \M_0$ (in the sense of Hausdorff).
We prove \eqref{KC} for $\M_0$, near $\M_0\cap\{x_{n+1}=0\}$. 
The proof also applies to $\M_k$ for $k$ sufficiently large, which leads to a contradiction.
\bei
\item [(I)]  
If $\M_0$ is strictly convex at $0$, then by \eqref{a2.15} we have $K_{\M_0}(p)<C$ for $p$ near the origin. 
$\M_0$ may not be smooth but we can apply \eqref{a2.15} to $\M_k$ and pass the estimate to the limit.

\vskip5pt

\item [(I\!I)] 
If $\M_0$ is not strictly convex at $0$,
denote $T=: \M_0\cap\{x_{n+1}=0\}$. 
Then $T$ is a convex set of dimension $d\ge 1$.

\vskip3pt

\bei
\item [(I\!I.1)] 
If for a small constant $h_0>0$, $\M_0\cap\{x_{n+1}<h_0\}$ can be represented as a graph
with bounded gradient,  then by \eqref{a2.15} we have $K_{\M_0}(p)<C$
for $p\in \M_0\cap\{x_{n+1}<h_0\}$.

\vskip3pt

\item [(I\!I.2)]  If $\M_0\cap\{x_{n+1}<h_0\}$ can not be represented as a graph with bounded gradient, 
we first derive the estimate near the extreme points of $T$.
Let $q$ be an extreme point.
First we make a translation of the coordinates such that $q=0$,
then by convexity, we can make a rotation of the coordinates in $\R^{n+1}$
such that in the new coordinates $y=(y_1, \cdots, y_n, y_{n+1})$,
$\M_0\cap\{y_{n+1}<h_0\}$ can be represented as a graph of a convex function $u$ with bounded gradient.
Since $\M_0$ is convex and normalized, the upper bound of the gradient $Du$ depends only on $n$.
We then apply the a priori estimate \eqref{a2.15} in Lemma \ref{L2.3} to conclude 
the estimate \eqref{KC} near $q$.

\vskip3pt

\item [(I\!I.3)]  From (I\!I.2), the Gauss curvature is bounded near all extreme points of $T$.
Denote by $E$ the set of all extreme points 
and let 
\vskip3pt
\centerline{\ \hskip50pt $E_\delta=\{p\in\M_0\ |\ \text{dist}(p, E)<\delta\}$, }
\vskip3pt
for a small $\delta>0$.
Let $\tilde q\not\in E_\delta$ be a boundary point of $T$.
Similarly as in (I\!I.2), we can choose an appropriate coordinates 
such that locally $\M_0-E_\delta$ is a graph of convex function $u$, 
and the gradient of $u$ is bounded.
We then apply Lemma \ref{L2.3a} to $u$ to obtain the estimate \eqref{a2.15}.
As a result, we obtain an upper bound for $K_{\M_0}$ near $\tilde q$.

\vskip3pt

\item [(I\!I.4)]  Denote by $\tilde E$ the boundary of $T$ and 
let  $\tilde E_\delta=\{p\in\M_0\ |\ \text{dist}(p, \tilde E)<\delta\}$.
By (I\!I.2) and (I\!I.3), we then apply Lemma \ref{L2.3a} to $\{\M_0-\tilde E_\delta\}\cap\{x_{n+1}<h_0\}$,
and obtain an upper bound for $K_{\M_0}$ in $\M_0\cap\{x_{n+1}<h_0\}$.
\eei

\noo
We have therefore obtained an upper bound for $K_{\M_0}(p)$ for all $p\in \M_0\cap\{x_{n+1}<h_0\}$.
By the convergence $\M_k\to \M_0$,
the above proof implies a uniform upper bound for $K_{\M_k}$ in $\M_k\cap\{x_{n+1}<h_0\}$.
But by our assumption, we have $K_{\M_k} (p_k)\ge k$ at  $p_k\in \M_k\cap\{x_{n+1}<\frac 1k\}$.
This is a contradiction.\hfill $\square$

\eei

\vskip5pt

Using the $C^{2, \alpha}$ estimate of Caffarelli for the Monge-Amp\`ere
equation \cite{[C2]}, and the $C^\alpha$ estimate of Caffarelli-Gutierrez
for the linearized Monge-Amp\`ere equation \cite{[CG]}, 
we have the following a priori estimates for the affine mean curvature equation \cite{[TW2]}.
See also \cite {[CW], [Le17]} where the conditions on $f$ can be relaxed.

\begin{Proposition}  \label{P2.1}  
Let $u\in C^4(\Om)$ be a locally uniformly convex solution of \eqref{a1.1}.
\newline
(i) Assume $f\in L^\infty(\Om)$. Then for any $p>1$ and $\Om'\subset\subset \Om$,
we have the estimate
\beq  \label{a2.18}  
\|u\|_{W^{4, p}(\Om')}\le C ,
\eeq
where $C$ depends on $n, p, \sup_\Om |f|$, $\text{dist}(\Om', \pom)$, and the
modulus of convexity of $u$.
\newline
(ii) Assume $f\in C^\alpha(\Om)$ for some $\alpha\in (0,1)$. Then
\beq  \label{a2.19}  
\|u\|_{C^{4, \alpha}(\Om')}\le C ,
\eeq
where $C$ depends on $n, \alpha, \|f\|_{C^\alpha(\Om)}$,
$\text{dist}(\Om', \pom)$, and the modulus of convexity of $u$.
\end{Proposition}

\subsection{The second boundary value problem} 

For our penalty argument for the approximation result, 
we  need the existence and regularity
of solutions to the second boundary value problem
\beqn
L[u] & =& f (x, u)  \ \ \ \  \text{in}\ \ \Om ,   \label{a2.20}    \\
   u & = & \phi\ \ \ \  \text{on}\ \ \pom ,  \no \\
   w & =& \psi\ \ \ \  \text{on}\ \ \pom , \no 
   \eeqn
where $L$ is the operator in \eqref{a1.1}. Assume that $\Om$ is a smooth,
uniformly convex domain in $\R^n$, $\phi, \psi$ are smooth functions
on $\pom$ and
\beq  \label{a2.21}  
C_0^{-1} \le \psi \le C_0      
\eeq
for some positive constant $C_0$. 
Assume also that  $f\in L^\infty(\Om\times\R)$ and $f$ is non-decreasing in $u$. 
Then we have  \cite{[TW2], [TW3]}

\begin{Proposition}   \label{P2.2}  
Under the above assumptions, problem \eqref{a2.20}
admits a unique solution $u\in W^{4,p}(\Om)$ $(\forall\ p>1)$. 
If $f\in C^\alpha(\bom\times \R)$, where $\alpha\in (0, 1)$, then the
solution $u\in C^{4,\alpha}(\bom)$.
\end{Proposition}

In \cite{[TW2]} we also assume that $f(x, t)\le 0$ when $t\le t_0$, for some $t_0<0$. 
This condition was removed in \cite {[CW]}.

\section{The affine Bernstein problem}

Consider the affine maximal surface equation  
\beq\label{a3.1}
U^{ij} w_{ij}(x)=0,   
\eeq
where $x=(x_1, \cdots, x_n)\in\R^n$, $w=[\det D^2 u]^{\frac{1}{n+2}-1}$,
$U^{ij}$ is the cofactor matrix of $D^2 u$.
In this section, we provide a counterexample to the affine Bernstein problem.

\begin{Theorem} \label{T3.1}
When the dimension $n\ge 10$,
there exists an entire smooth convex solution to equation \eqref{a3.1}
which is not quadratic polynomial.
 \end{Theorem}
 
Our candidate for the counter-examples has the following properties
\bei
\item [(i)] $u\ge 0$. 
\item [(ii)] $u$ is rotationally symmetric in $x'=(x_1, \cdots, x_{n-1})$.
\item [(iii)] $u$ is symmetric in $x_n$.
\item [(iv)] $u(0, x_n)\ge |x_n|$ and $\lim_{x_n\to\infty} \frac{u(0, x_n)}{|x_n|}=1$. 
\item [(v)] If $(x_1, \cdots, x_n, x_{n+1})\in \R^{n+1}$ is a point on the graph of $u$, then 
$$
(x_{n+1}+x_{n})(x_{n+1}-x_{n}) = 2 h(x'), 
$$
for a function $h(x')$. 
Hence $u$ is given by
\beq\label{a3.2}
u(x)= \sqrt{x_n^2 +2h(|x|)} .
\eeq
\eei

If we rotate the $(x_n, x_{n+1})$-plane by 45 degree, 
namely if we make the coordinate change
$$
{\begin{split}
  & y'=x',\\
  & y_n=(x_n+x_{n+1})/\sqrt 2,\\
  & y_{n+1}=(x_{n+1}-x_n)/\sqrt 2 , 
  \end{split}} $$
then the graph of $u$ can be expressed as 
$$y_ny_{n+1}=h(x') . $$
By a slight abuse of notation, we still use $u$ to denote the function $\frac{h(y')}{y_n}$.
From the affine invariance, $u$ satisfies equation \eqref{a3.1}.

\vskip10pt

\subsection{Derivation of the ODE} Denote $r=|y'|$, $t=y_n$. Then we have
\beq\label{a3.3}
u=\frac{h(r)}{t} ,
\eeq
We compute
$${\begin{split}
  & u_r=\frac{h'}{t},\ \ \ u_t= -\frac{h}{t^2},  \\
 &  u_{rr}=\frac{h''}{t},\ \ u_{rt}=-\frac{h'}{t^2}, \ \ u_{tt}=\frac{2h}{t^3}. 
  \end{split}} 
  $$
  and in polar coordinates, 
$$
D^2 u = {\begin{pmatrix}
 & \frac{u_r}{r} , &\cdots, &0, & 0\\
 &\cdot, &\cdots, &\cdot, & \cdot\\
 &0, &\cdots, &u_{rr}, & u_{tr}\\
 & 0, & \cdots, &u_{tr}, &u_{tt}\\
\end{pmatrix}} . $$
Denote
$$\hskip0pt {\begin{split}
   \Delta &= u_{rr}u_{tt}-u_{tr}^2=\frac{1}{t^4} [2hh''-h'^2] , \\
   D &=\det D^2 u =\Big(\frac{u_r}{r}\Big)^{n-2} \Delta  \\
     & = \frac{1}{t^{n+2}} \Big(\frac{h'}{r}\Big)^{n-2} [2hh''-h'^2],
\end{split}} $$
and
$$\hskip0pt {\begin{split}
   w &= D^{\frac{1}{n+2}-1}=t^{n+1} \Big(\frac{h'}{r}\Big)^{-\frac{n+1}{n+2}(n-2)}  \big[2hh''-h'^2\big]^{-\frac{n+1}{n+2}}\\
       &  =: t^{n+1} g.\\
\end{split}} $$
The cofactor matrix is
$${\begin{split}
U^{ij} &=
 {\begin{pmatrix}
  \big(\frac{u_r}{r}\big)^{n-3}\Delta,  &\cdots, &0, & 0\\
 \cdot, &\cdots, &\cdot, & \cdot\\
 0, &\cdots, &\big(\frac{u_r}{r}\big)^{n-2}u_{tt}, & -\big(\frac{u_r}{r}\big)^{n-2}u_{rt}\\
 0, &\cdots, &-\big(\frac{u_r}{r}\big)^{n-2}u_{rt}, &\big(\frac{u_r}{r}\big)^{n-2}u_{rr}\\
 \end{pmatrix}} \\
 & =\big(\frac{u_r}{r}\big)^{n-2} {\begin{pmatrix}
   \big(\frac{u_r}{r}\big)^{-1}\Delta,  &\cdots, &0, & 0\\
   \cdot, &\cdots, &\cdot, & \cdot\\
 0, &\cdots,&u_{tt}, &-u_{tr} \\
 0, &\cdots, &-u_{rt}, &u_{rr}\\
\end{pmatrix}}  . 
 \end{split}}$$
We compute
$${\begin{split}
\big(\frac{u_r}{r}\big)^{2-n} U^{ij}w_{ij}
 &= u_{rr}w_{tt}-2u_{rt}w_{rt}+u_{tt}w_{rr} +(n-2)\big(\frac{u_r}{r}\big)^{-1} \Delta \frac{w_r}{r} \\
 &= u_{rr}w_{tt}-2u_{rt}w_{rt}+u_{tt}w_{rr} +(n-2) \Delta \frac{w_r}{u_r} \\
 &= n(n+1)t^{n-2}h''g+2(n+1)t^{n-2} h'g'+2t^{n-2}hg''+(n-2)t^{n-2}\frac{g'}{h'} [2hh''-h'^2] .
 \end{split}} $$
Hence the affine maximal surface equation \eqref{a3.1} becomes
$$
n(n+1)h''g+2(n+1) h'g'+2hg''+(n-2)\frac{g'}{h'} [2hh''-h'^2] =0  , 
$$
namley
$$
n(n+1)h''g+(n+4) h'g'+2hg''+2(n-2)\frac{hh''}{h'} g'=0  
$$
It can be written as
\beq\label{a3.4}
g''+\Big[\frac{n+4}{2}\frac{h'}{h}  +(n-2)\frac{h''}{h'}\Big] g' +\frac{n(n+1)}{2} \frac{h'}{h} \frac{h''}{h'} g=0 . 
\eeq
where
\beq\label{a3.5}
g=\Big(\frac{h'}{r}\Big)^{-\frac{n+1}{n+2}(n-2)}  \big[2hh''-h'^2\big]^{-\frac{n+1}{n+2}} . 
\eeq
Regarding the initial condition,
given that $h$ admits a smooth even extension about $r = 0$, we must have 
$\frac{d^k}{dr^k} h(0)=0$ for all odd $k$.
From the convexity of $h$,  we assume that  $h(0)> 0$ and $h''(0)> 0$. 
Hence we may assume that
\beq\label{a3.6}
h(0) = 1,\quad h'(0) = 0,\quad h''(0) = 1, \quad h'''(0) = 0.
\eeq
Note that if $h(0)=0$,  from properties (i)-(v), 
the lines $\{(x', x_n, x_{n+1})\in\R^{n+1}\ |\ x'=0, x_{n+1}=|x_n| \}$ are contained in the
graph of $u$, hence $u$ is not smooth.
If $h(0)>0$, $u$ is strictly convex near the origin,  
by the regularity of the affine maximal surface equation, 
$u$ is smooth near the origin.

The local existence and uniqueness of solutions to the ODE \eqref{a3.4} - \eqref{a3.6}
is guaranteed by the contraction mapping theorem, since a higher order ODE
can be written as a system of first order ODEs.
We need to prove that the solution exists for all $r>0$. 
However, the ODE \eqref{a3.4} is a fourth order, strongly nonlinear,
it is difficult to work on it directly.

\subsection{A counterexample in dimension n=10.}
Here we use the affine volume functional $A$ to prove the existence of a complete convex solution
to the ODE \eqref{a3.4}, for dimension $n=10$, using the non-smooth counter-example in \cite{[TW1]}
as a sub-solution.

By the local existence of solutions, there is a solution $h$ to \eqref{a3.4} for $r>0$ small.
Assume the maximal existence interval for the solution is $r\in [0, R^*)$, for some $R^*>0$.
We need to prove that $R^*=\infty$.

Assume to the contrary that $R^*<\infty$.
There are  three possibilities.
\begin{itemize}
\item [] {\it Case 1)}: \ $h(R^*)=\infty$,
\item [] {\it Case 2)}: \ $h(R^*)<\infty$ but $h'(R^*)=\infty$,
\item [] {\it Case 3)}: \  $h(R^*)<\infty$ and $h'(R^*)<\infty$.
\end{itemize}

Both {\it Case 1)} and {\it Case 2)} can be ruled out easily.
It was found in \cite{[TW1]} that the function $h^*(r)=c^* r^9$ 
satisfies \eqref{a3.4} in dimension $n=10$, 
and the corresponding function $u(x)=\frac{h^*(|y'|)}{y_n}$  in \eqref{a3.3}  is affine maximal.

When $c^*$ is small, we have $h^*(r)<h(r)\ \forall\ r\in [0, R^*)$.
We can increase the constant $c^*$ until the point  such that
\begin{itemize}
\item  [$\diamond$] $h^*(r)\le h(r)$ for $r\in (0, R^*)$ and 

\item  [$\diamond$] there exists $r_0\in (0, R^*)$ such that 
$h^*(r_0) = h(r_0)$, and  $(h^*)'(r_0) = h'(r_0)$.
 \end{itemize} 
Hence $u(y)=\frac{h(|y'|)}{y_n}$ and $u^*(y)=\frac{h^*(|y'|)}{y_n}$ (see \eqref{a3.3})
are both maximizers of the affine volume functional $A$.

For a convex function $u$ given by \eqref{a3.3}, we have
$$   \det D^2 u  = \frac{1}{t^{n+2}} \Big(\frac{h'}{r}\Big)^{n-2} [2hh''-h'^2]. $$
Hence the affine volume functional is 
$$ {\begin{split}
A(u) &=\int_\Om  \big[\det D^2 u \big]^{\frac{1}{n+2}} \\
        &=\int_\Om \frac 1t \Big(\frac{h'}{r}\Big)^{\frac{n-2}{n+2} } \big[2hh''-h'^2\big]^{\frac{1}{n+2}},
\end{split}} $$
Thus $u$ is affine maximal if and only if $h$ is maximal with respect to the functional  
$$
\mathcal J(h)
 = \int_{- R}^{ R} r^{n-2} \Big(\frac{h'}{r}\Big)^{\frac{n-2}{n+2} } \big[2hh''-h'^2\big]^{\frac{1}{n+2}} dr . 
$$
It follows that both  $h$ and $h^*$ are maximizers of the functional $\mathcal J$ in $\ol S[\mathcal I_0; \eta, \kappa]$,
where $\mathcal I_0=(-r_0, r_0)$, $\eta=h(r_0)$, $\kappa=h'(r_0)$.
This yields a contradiction with the uniqueness of the maximizers of the concave functional $\mathcal J$.
Note that the concavity of $\mathcal J$ is inherited from that of $A$, by restricting $A$ to the domain
$\Om=\{y\in\R^n\ |\ |y'|<r_0, 1<y_n<2\}$.

By $\ol S[\mathcal I_0;  \eta, \kappa]$ we denote
the set of symmetric, positive  functions $h$ defined on $\mathcal I_0$ satisfying
$h(R)=\eta$ and $h'(R)\le \kappa$, such that $u(y)=\frac{h(|y'|)}{y_n}$ is convex.

\vskip5pt

In {\it Case 3)}, we prove that the solution can be smoothly extended to a small neighbourhood of $R^*$.
This contradicts to the assumption that $[0, R^*)$ is the maximal existence interval 
for the initial value problem \eqref{a3.4}-\eqref{a3.6}. 
For this purpose, it suffices to prove that under the conditions 
 $h(R^*)<\infty$ and $h'(R^*)<\infty$ in Case 3),  
 the quantities $h''(R^*)$ and $h'''(R^*)$ are also finite.
If this is done, the local existence theorem to ODEs implies that 
the solution $h$ can be extended to a neighbourhood beyond $R^*$.
Therefore, it suffices to prove the following Lemma \ref{L3.2}.

\subsection{Estimates for the ODE \eqref{a3.4}}
The following Lemma \ref{L3.2} addresses the boundary regularity of solutions of the form \eqref{a3.3} 
to the first boundary value problem for the affine maximal surface equation.


\begin{Lemma}\label{L3.2}
Let \(h(r)\) be a smooth, positive, convex function defined for \(r \in [0,1]\) 
satisfying equation \eqref{a3.4} and the initial conditions \eqref{a3.6}.
Assume that there exist constants \(C_1, C_2 > 0\) such that
\[
0 < h(1) \le C_1, \quad 0 < h'(1) \le C_2.
\]
Then there exist constants \(C_3, C_4 > 0\) such that
\beq\label{a3.7}
h''(1) \le C_3, \quad |h'''(1)| \le C_4.
\eeq
\end{Lemma}

\noindent{\bf Remark 3.1.}\ 
Once \eqref{a3.7} is established,
estimates for higher derivatives follows from the equation \eqref{a3.4} directly.
The constants $C_3, C_4$ depend on \(n\), \(C_1\), and \(C_2\).
Since $h$ is a smooth, convex solution to \eqref{a3.4},
there exists constant $C^*>0$ such that
$$
h''(r)<C^*,\  \ h''''(r)\le C^* , \ \  h''(r)\ge 1/C^*\ \ \forall\ \ r\in [0, {\Small\text{$\frac 12$}}],
$$
and $h'(r)\ge r/C^*\ \forall\ r\in [0, \frac 12].$ 
We also allow that $C_3$ and $C_4$ depend on $C^*$.

\begin{proof}
Denote 
$$S(r) =  2h(r)h''(r) - (h'(r))^{2} , $$
$$p(r) =  h(r)^{\frac{n + 4}{2}}\left(h'(r)\right)^{n - 2} .$$
A direct computation gives
\[
\frac{p'(r)}{p(r)} = \frac{n + 4}{2}\frac{h'(r)}{h(r)} +(n - 2)\frac{h''(r)}{h'(r)}.
\]
Hence the ODE \eqref{a3.4} is equivalent to
\beq\label{a3.8}
(pg')'(r) + \frac{n(n + 1)}{2}\frac{h''(r)}{h(r)} p(r)g(r) = 0.  
\eeq
Since $h''\geq 0$, $h\geq 1$, $p > 0$ and $g > 0$, we have $(pg')'\leq 0$, so $pg'$ is non-increasing on $(0,1]$. 
Because $h$ is even in $r$, both $\frac{h'}{r}$ and $S(r)$ are even and smooth, 
hence \eqref{a3.5} shows that $g$ is even; in particular $g'(0) = 0$. 
Therefore $pg'\leq pg'(0) = 0$ and
\beq\label{a3.9}
g'(r)\leq 0\qquad \forall \ r\in (0,1].  
\eeq
Hence $g$ is decreasing and $0< g(1)\leq g(r)\leq g(0)$.

Write \eqref{a3.4} in the form
\beq\label{a3.10}
g''(r) + A(r)g'(r) + B(r)g(r) = 0, \quad 
\eeq
where
\[
A(r) = \frac{n + 4}{2}\frac{h'(r)}{h(r)} +(n - 2)\frac{h''(r)}{h'(r)},\qquad B(r) = \frac{n(n + 1)}{2}\frac{h''(r)}{h(r)}.
\]
The integrals of $A$ and $B$ admit bounds depending only on $n,C_1,C_2$ and $C^*$:
\begin{align*}
\int_{r_0}^1\frac{h'}{h} dr &= \log h(1) - \log h(r_0)\leq \log C_1, \\
\int_{r_0}^1\frac{h''}{h'} dr &= \log h'(1) - \log h'(r_0)\leq \log \frac{C_2}{h'(r_0)}, \\
\int_{r_0}^1\frac{h''}{h} dr &\leq \int_{r_0}^1h''(r)dr = h'(1) - h'(r_0)\leq C_2.
\end{align*}
where $r_0\in [\frac 14, \frac 12]$. We may fix $r_0=\frac 12$.
Hence
\beq\label{a3.11}
\int_{r_0}^1A(r)dr + \int_{r_0}^1B(r)dr\leq L = L(n,C_1,C_2, C^*), 
\eeq
for a finite constant $L$.

Now we need the following lemma.

\begin{Lemma}\label{L3.3}
There exists $K = K(L)\geq 1$ such that
\beq\label{a3.12}
K^{-1}\big(g(r_0) + |g'(r_0)|\big)\leq g(1) + |g'(1)|\leq K\big(g(r_0) + |g'(r_0)|\big). 
\eeq
\end{Lemma}

\begin{proof}
Set $Y = (g,g')^{\mathrm{T}}$. Then $Y' = M(r)Y$ with
\[
M(r) = \begin{pmatrix}0 & 1\\ -B(r) & -A(r)\end{pmatrix}.
\]
Since $A,B\geq 0$, we have $\| M(r)\| \leq 1 + A(r) + B(r)$ and thus, by \eqref{a3.11},
\[
\int_{r_0}^1\| M(r)\| dr\leq 1 + L =: L_*.
\]
Let $\Phi$ be the fundamental matrix with $\Phi (r_0) = I$. 
From the Grönwall inequality, $\| \Phi (1)\| \leq e^{L_*}$.

Also 
$$(\det\Phi)^{\prime} = \operatorname{tr}(M)\det\Phi = - A\det\Phi , $$ 
hence
\[
\det\Phi (1) = \exp \Bigl(-\int_{r_0}^1A(r)dr\Bigr)\geq e^{-L}.
\]
In dimension 2, 
$$\| \Phi (1)^{- 1}\| \leq C\| \Phi (1)\| /\det\Phi (1) \leq Ce^{L_*}e^L , $$ 
where $C$ is an absolute constant. Therefore,
\[
\| Y(1)\| \leq \| \Phi (1)\| \| Y(r_0)\| \leq e^{L_*}\| Y(r_0)\|,
\]
\[
\| Y(1)\| \geq \| \Phi (1)^{-1}\|^{-1}\| Y(r_0)\| \geq C^{-1}e^{-(L_* + L)}\| Y(r_0)\| .
\]
where $\| Y\| = g + |g'|$. 
We obtain \eqref{a3.12} with $K = K(L)$.
\end{proof}

We proceed to obtain a positive lower bound for $g(1)$.
Let $h=\phi^2$, then
\beq\label{a3.13}
{\begin{split}
g & = \phi^{-\frac{(n+1)^2}{n+2}} \Big(\frac{\phi'}{r}\Big)^{-\frac{n+1}{n+2}(n-2)}  \big[\phi''\big]^{-\frac{n+1}{n+2}} \\
 & = r^{-\beta} \phi^\gamma \phi'^\beta \phi''^\alpha \\
\end{split}}\eeq
where we ignore a coefficient $2^{-\frac{(n+1)^2}{n+2}}$ and temporarily denote 
$\alpha=-\frac{n+1}{n+2}$, $\beta= -\frac{n+1}{n+2}(n-2)$, $\gamma= -\frac{(n+1)^2}{n+2}$.
Differentiating the above formula,  we have
\beq\label{a3.14}
{\begin{split}
g'(r)= & \alpha \phi'''\phi''^{\alpha-1} \phi'^\beta \phi^\gamma r^{-\beta} 
+\beta \phi''^{\alpha+1} \phi'^{\beta-1} \phi^\gamma r^{-\beta}\\
 &+ \gamma \phi''^{\alpha} \phi'^{\beta+1} \phi^{\gamma-1} r^{-\beta}
 -\beta \phi''^{\alpha} \phi'^\beta \phi^\gamma r^{-\beta-1} .
 \end{split}} \eeq
By Lemma \ref{L3.3},   $|g'(r)| $ is bounded for $r$ near 1. 
Hence
$$ \phi'''\phi''^{\alpha-1}  
+\hat \beta (r) \phi''^{\alpha+1} 
+ \hat \gamma(r) \phi''^{\alpha} =\hat\delta(r).$$
where 
$\hat\beta(r)=\frac{\beta}{\alpha \phi'} $, 
$\hat \gamma(r)=\frac{\gamma \phi'}{\alpha\phi}-\frac{\gamma}{\alpha r},$
and $\hat \delta(r)=\frac{g'(r) r^\beta}{\alpha \phi'^\beta \phi^\gamma  }$.
We write the above formula as 
\beq\label{a3.15}
\phi'''
+\hat \beta (r) \phi''^2 
+ \hat \gamma(r) \phi'' =\hat\delta(r)\phi''^{1-\alpha}.\eeq
Note that $\hat \beta$ is positive and bounded.
Hence $2\beta_0:=\inf\{\hat\beta(r)\ |\ \frac 12 <r<1\}>0$.
Note that $1-\alpha<2$, $|\hat\gamma(r)|$ and $|\hat\delta(r)|$ is bounded for $\frac 12\le r\le 1$.
We obtain the inequality
$$ \phi''' +\beta_0 \phi''^2 \le C_0\ \ \ \forall\ r\in ({\Small\text{$\frac 12$}}, 1). $$
It follows that
$$
\phi''(1)\le \phi''({\Small\text{$\frac 12$}})+{\Small\text{$\frac 12$}}C_0.
$$
In the above estimate, the number 1 can be replaced by any number $r\in [\frac 12, 1]$.
Recall that $h=\phi^2$, hence
\beq\label{a3.16}
h''(r)\le C^*\ \ \forall\ r\in [{\Small\text{$\frac 12$}}, 1]. 
\eeq

Estimate \eqref{a3.16} implies that $g$ has a positive lower bound for $r\in [\frac 12, 1]$.
To obtain an upper bound for $|h'''(r)|$ for $r\in [\frac 12, 1]$,
it suffices to differentiating $g$  as in \eqref{a3.14} to derive \eqref{a3.15}.
This completes the proof of Theorem 1.1.
\end{proof}

\noindent{\bf Remark 3.2.}\ 
When $n=10$,  the function $h^*(r)=r^9$ plays a crucial role in the proof.
For dimension $n\le 9$, there is no such function and we cannot rule out 
the scenario of Case 2).

\vskip5pt
 
\noindent{\bf Remark 3.3.}\ 
Once a counterexample $u$ in dimension 10 is obtained, 
a function of the form $u(\tilde x)+\frac 12 |\hat x|^2$ provides a
counterexample in higher dimensions, where
$\tilde x=(x_1, \cdots, x_{10})$ and $\hat x=(x_{11}, \cdots, x_n)$.

\vskip5pt

\subsection{\bf A Bernstein type theorem}

By the counterexample in Theorem \ref{T3.1}, 
we need to impose superlinear growth condition 
to insure that an entire convex solution to equation \eqref{a3.1} is a quadratic function,
at least for dimensions $n\ge 10$. 

\begin{Theorem} \label{Thm3.2}
Let $u$ be  an entire smooth convex solution to \eqref{a3.1}.
Then $u$ is a quadratic polynomial
if there exists a positive constant $\delta>0$ such that either 
\beq \label{grow-a}
u(\lambda x)\ge (\lambda+\delta)u(x) \ \ \text{as}\ \ |x|\to\infty,
\eeq
for a fixed constant $\lambda>1$; or
\beq \label{grow-b}
x\cdot Du(x)\ge (1+\delta) u(x)\ \ \text{as}\ \ |x|\to\infty .
\eeq
 \end{Theorem}

As mentioned before, condition \eqref{grow-b} was inspired by \cite{[EH90]}.
We point out that condition \eqref{grow-b} implies  \eqref{grow-a}.
Indeed,  let $x=te$ and denote $g(t)=u(re)$, where $e$ is a unit vector.
Then  \eqref{grow-b} implies that  $tg'\ge (1+\delta) g$. Hence 
$$\int_a^b \frac{g'}{g}\ge \int_a^b \frac{1+\delta}{t} dt.$$
Hence $g(b)\ge \big(\frac ba\big)^{1+\delta} \, g(a)$.
Let $b=\lambda a$ and $a=|x|$, we obtain \eqref{grow-a}.

\vskip5pt

\begin{proof}[Proof of Theorem \ref{Thm3.2}] 
Condition \eqref{grow-a} implies $u$ is superlinear and $Du(\R^n)=\R^n$.
Let $v(y)=:u^*(y)$ be the Legendre transform of $u$. Then $v$ satisfies
\beq\label{LT}
{\begin{split}
 & V^{ij}\omega_{ij}(y)=0\ \ \text{in}\ \R^n,\\
 & \omega=\big[\det D^2 v\big]^{-\frac{1}{n+2}} .
 \end{split}} \eeq
Choose a sequence $h_k\to \infty$ as $k\to\infty$. 
Let $\ell_k$ be a linear function such that 
$$S^c_{0, h_k}[v]=\{y\in\R^n\ |\ v(y)<\ell_k(y)\}$$
is a convex domain with its mass centre at the origin.
The existence of such a linear function $\ell_k$ was proved in \cite{[AC], [C92]}. 
For brevity, denote $S_k=S^c_{0, h_k}[v]$.

Let $T_k: y\to z$ be a linear transform such that the minimum ellipsoid of  $T_k(S_k)$ is the unit ball,
and define 
$$v_k(z)=h_k^{-1} v(T_k^{-1} z) .$$
Then 
\bei
\item [a)] $v_k(0)=1$,
\item [b)] $0\le v_k\le c_n$ in $T_k(S_k)$,  
\item [c)] $v_k=h_k^{-1}\ell_k$   on $\p T_k(S_k)$,
\item [d)] $v_k$ satisfies equation \eqref{LT}.
\eei
where the upper bound in (b) is due to that $T_k(S_k)$ is normalized.
Denote
$$\M^*_k=\{(z, v_k(z))\in\R^{n+1}\ |\ z\in \R^n\}$$
be the graph of $v_k$.
Let $Q^*_k$ be the convex body enclosed by $\M^*_k$ and the hyperplane  $h_k^{-1}\ell_k$.
By the above normalization, there exists $R>r>0$, depending only on $n$, such that 
$$B_r({\Small\text{$\frac 12$}} e_{n+1})\subset Q^*_k\subset B_R({\Small\text{$\frac 12$}}  e_{n+1}), $$
where $B_r(\frac 12 e_{n+1})$ is a ball in $\R^{n+1}$ with radius $r$ and center $\frac 12 e_{n+1}$).
 
Therefore, by passing to a subsequence, we may assume that $T_k(S_k)$ converges to a convex domain $\Om_0$
and  $\M^*_k$ converges to a hypersurface $\M^*_0$ (in the sense of Hausdorff).

Restrict $v_k$ to $\Om_0$. 
By our normalization, $v_k$ is locally uniformly bounded in $\Om_0$.
By convexity,  $v_k$ converges locally uniformly to a convex function $v_0$ in $\Om_0$.

We claim that the gradient of $v_0$ is uniformly bounded in $\Om^*_0$.

If the claim is not true, then there is a boundary point 
$y^*\in \pom^*_0$, and a vertical line 
$$L^*=\{y^*+te_{n+1}\ |\ t\ge t^*\ge 0\} ,$$ 
such that $L^*$ is contained in $\M^*_0$.
By a linear transform of the coordinates, there is no loss of generality in assuming that 
\beq\label{vL}
\text{$y^*=e_1$ and $\Om^*_0\subset\{y_1<1\}$. }
\eeq

Let $u_0$ and $u_k$ be the Legendre transform of $v_0$ and $v_k$.
Then $u_0$ and $u_k$ are affine maximal, 
$u_0$ is the limit of $u_k$, and $u_k$ is smooth and locally uniformly convex.
Note that the condition \eqref{grow-a} is affine invariant.
Hence $u_0$ and $u_k$ also satisfy \eqref{grow-a}.

By \eqref{vL}, $u_0(te_1)$ is asymptotically linear for $t>0$. 
But  \eqref{grow-a} implies that $u_0(te_1)$ is superlinear in $t$.
We reach a contradiction and the claim follows.

Therefore by Lemma \ref{L2.3}, we obtain 
$\det D^2 v_0<C_1.$
On the other hand, by the a priori estimate for $u_0$  (Lemma \ref{L2.4}), 
and by the Legendre transform, we also have $\det D^2 v_0\ge C_2>0$.
By Lemma \ref{L2.2}, $v_0$ is strictly convex and $C^1$ smooth.
By the regularity theory for the Monge-Amp\`ere equation \cite{[C2]}, 
$v_0$ is uniformly convex and $C^\infty$ smooth.

By a blow-down argument (see the proof of Theorem 2.1 in \cite {[TW1]}),
we conclude that $v$ is a quadratic polynomial.
\end{proof}

\section{Approximation by smooth solutions}

Let $u\in\ol S[\phi, \Om]$ be the  maximizer of \eqref{a1.9}. 
In \cite[Theorem 6.1]{[TW2]}, we proved that $u$ can be approximated by smooth affine
maximal functions. In this section we prove a stronger
version of approximation, under very mild conditions on $\phi$ and $\Om$. 
This stronger version is critical in the proof of strict convexity for the affine maximal hypetrsurfaces.

Let $\Om$ be a bounded, open set in $\R^n$ with $|\pom|=0$. 
Let $\phi$ be a convex, uniformly Lipschitz function defined in
$\Om^{\delta}$ for some $\delta>0$, satisfying \eqref{phib}.
Let $B=B_R(0)$ be a ball such that $B_{R-1}(0)$ contains $\Om^{\delta}$.

In our proof below we need a convex function $\phi_*$ in $\bom$, given by
\beq  \label{a4.1}  
\phi_*(x)=\sup\{\L(x)\ |\ \L\ \text{is a supporting plane of $\phi$
                           at some point   on $\p\Om$}\, \}.
\eeq
The function $\phi_*$ is convex, as it is the supremum of a family of affine 
functions. Since $N_{\phi_*}(\Om)\subset \overline{N_\phi(\Om)}$, we have
$\phi_*\in\ol S[\phi, \Om]$. On the other hand, for any function
$u\in\ol S[\phi, \Om]$, by the convexity of $\phi$, we have $u\ge \L$
for any supporting plane $\L$ of $\phi$ at some point in $\p\Om$. Hence
\beq  \label{a4.2}  
u\ge \phi_*\ \ \text{in}\ \ \Om
\eeq
for any $u\in\ol S[\phi, \Om]$.

Since $\phi$ is uniformly Lipschitz in $\Om^\delta$, 
we can choose a large constant $c^*$ such that any
supporting plane $\L$ of $\phi$ in $\Om^\delta$ satisfies $\L\le c^*$ on
$\p B$. Therefore we can extend $\phi$ to $B$ such that 
$\phi$ is convex in $B$ and $\phi=c^*$ on $\p B$. 
The extension is given by
\beq  \label{a4.3}  
\phi(x)=\sup\{\L(x)\ |\ \L\ \text{is an affine function},
   \ \L\le \phi\ \text{in}\ \Om^\delta
      \ \text{and}\ \L\le c^*\ \text{in}\ B\}.
\eeq
The Monge-Amp\`ere measure $\det D^2 \phi$ vanishes in $B-\overline{\Om^\delta}$, and may have singular part 
on the boundary $\pom^\delta$.

Let
\beq  \label{a4.4}  
H(t) =(1-t^2)^{-2n}
\eeq
be a smooth function in the interval $(-1, 1)$. When $|t|\ge 1$, we
can define formally $H(t)=+\infty$. We call $H$ the penalty
function, as it forces the solution $u$ to stay close to $\phi$.
In \cite{[TW2]} we employed a different penalty function but \eqref{a4.4} 
is preferable here because  the closeness estimate \eqref{a4.17} follows
more directly.

We now extend the function $f$ in \eqref{a1.1} to $B\times\R$ so that
\beq  \label{a4.5}  
f(x,u)=\left\{ \begin{array}{ll}
f(x) \ \ \ & x\in\Om,\\
H'(u-\phi(x))\ \ \ \ & x\in B-\Om.
\end{array}  \right.  
\eeq
The following Lemma \ref{L4.1} corresponds to Lemma 6.1  in \cite{[TW2]}.

\begin {Lemma} \label{L4.1} 
Suppose $\pom$ is Lipschitz continuous. Then
there is a uniformly convex solution $u\in W^{4, p}(B)$ (for all $p<\infty$) 
to the boundary value problem
\beqn  \label{a4.6}  
L[u] & =& f(\cdot, u)\ \ \ \text{in}\ \ B,  \\
  u  & =&  \phi\ (=c^*)\ \ \ \text{on}\ \ \p B,\no\\
  w  & =& 1 \ \ \ \text{on}\ \ \p B .\no  
  \eeqn
\end{Lemma}

\noo{\it Proof}. 
Let $u$ be a solution of \eqref{a4.6}. 
It suffices to prove that
\beq  \label{a4.7}  
|f(x, u)|\le C
\eeq
for some constant $C>0$ independent of $u$. 
Once $f$ is bounded, the regularity of solutions follows from
Proposition \ref{P2.2}, and the existence of solutions can be obtained by a
degree argument, as in \S5 of \cite{[TW2]}.
The proof below uses ideas from that of Lemma 6.1 \cite{[TW2]}
but as the function $H$ here is different from the one in \cite{[TW2]},
some changes are necessary. 
For clarity we divide the proof into a few steps.

\bei

\item [1.] 
By our choice of the function $H$ in \eqref{a4.4}, 
we have $u>\phi-1$ in $B-\Om$. 
By the convexity of $u$ we then have
$$
u\ge -C\ \ \ \text{in}\ \ B.
$$

\vskip5pt

\item [2.] 
Since $u=\phi$ on $\p B$, by the convexity of $u$ and $\phi$, 
we have $u\le \phi +1/2$ in a neighborhood of $\p B$, namely
$u(x)\le \phi(x) +1/2$ when $R-c_0<|x|<R$, 
where $c_0$ depends only on $\sup_B |D\phi|$. 
This implies by our choice of $H$ that 
$$
f(x, u)\le C\ \ \ \text{if}\ \ R-c_0<|x|<R.
$$

\vskip5pt

\item [3.] 
In this step we  prove that 
\beq \label{a4.8}
\det D^2 u(x)\le C\ \ \ \text{if}\ \ R-\frac{c_0}{2}<|x|<R.
\eeq
Indeed,   for any boundary point $x_0\in \p B$, we suppose by a
rotation of axes that $x_0=(R, 0, \cdots, 0)$. Let
$\ell(x)=c^*-\frac 12-a(x_1-R)$ be a linear function such that
$\ell(x_0)=u(x_0)-\frac 12$.  
We choose the constant $a= 1/c_0$ such that 
$\ell(x)>c^*>u(x)$ when $x_1\le R-c_0$.
Let $z=\log \frac {w}{u-\ell}$, which is function defined in the set $\{\ell<u\}$. 
If $z$ attains its minimum at a boundary point $\p B$, by the boundary
condition $w=1$ in \eqref{a4.6} we see that $z\ge -C$, 
namely $z$ is bounded from below. 
If $z$ attains its minimum at an interior point $y_0\in \{u>\ell\}$, then 
\beqs
&0= & z_i=\frac {w_i}{w}-\frac {(u-\ell)_i}{u-\ell}\ \ \text{at}\ y_0,\\
   &0\le & z_{ij}=\frac {w_{ij}}{w}-\frac {(u-\ell)_{ij}}{u-\ell}\
\eeqs
as a matrix. It follows that,
$$0\le u^{ij}z_{ij}
         =\frac {f}{[\det D^2 u]^\th}-\frac {n}{u-\ell}, $$
where $\th=1/(n+2)$. We have shown that $f\le C$ near $\p B$.
Hence $\det D^2 u\le C$ at $y_0$ and so $z \ge -C$, which in turn
implies that $\det D^2 u<C$ near $\p B$.

\vskip5pt

\item [4.] 
In this step we show that 
\beq  \label{a4.10}  
f\ge -C\ \ \text{in}\ B.
\eeq
Since $-C\le u\le c^*$ in $B$, $u=c^*$ on $\p B$,  
and $u$ satisfies \eqref{a4.8}, 
by constructing barriers we have $|Du|\le C$ on $\p B$. 
Hence by convexity we have 
\beq  \label{a4.9}  
|Du|\le C\ \ \ \text{in}\ B. 
\eeq
Recalling that $u=\phi = c^*$ on $\p B$,
we have $u(x)\ge \phi(x)-\frac 12$ when $R-c_1<|x|<R$, 
for some $c_1$ depending only on $\sup_B |Du|$.
Hence \eqref{a4.10} holds when $R-c_1<|x|<R$.  
Next we show  \eqref{a4.10} holds in the whole domain $B$.

\vskip5pt

If \eqref{a4.10} is not true, there is a point $y_0\in B-\Om$ such that 
$$f(y_0, u(y_0))\le  \inf_B f+1 .$$
If $f(y_0, u(y_0))$ is negatively very large, by our choice of $H$ in \eqref{a4.4},
necessarily $u(y_0)-\phi(y_0)$ is very close to $-1$. 
Let $\ell$ be a tangent plane of $\phi$ at $y_0$. 
Then obviously $\ell(y_0)>u(y_0)+\frac 12$ and $\ell\le u$ on $\p B$.
We compute
\beqs
\int_{\U_\ell} U^{ij}w_{ij}(u-\ell)
 & =& -\int_{\U_\ell} (u-\ell)_i w_j U^{ij} \\
 & = &-\int_{\p\,\U_\ell} (u-\ell)_i\gamma_j w U^{ij}
     +\int_{\U_\ell} w U^{ij} (u-\ell)_{ij}\\
 & <& n\int_B w\det D^2 u =n\int_B [\det D^2 u]^\th, 
\eeqs
where $\U_\ell=\{x\in B\ |\ u(x)<\ell(x)\}$, 
$\gamma$ is the unit outward normal. Hence
\beq  \label{a4.11}  
\int_{\U_\ell} f(x, u)(u-\ell)\le n\int_{\U_\ell}  [\det D^2 u]^\th
 = n\int_{\M_\ell} K^\th \le C, 
\eeq
where $\M_\ell$ is the graph of $u$ over the set $\U_\ell$, 
and $K$ is the Gauss curvature of $\M_\ell$. 

\vskip5pt

Recall that $\ell$ is  the tangent plane of $\phi$ at $y_0$.
Hence in the set $\U_\ell$, we have $u<\ell<\phi$ and $H'(u-\phi)<0$,
which implies that $f(x, u)\le C$ in $\U_\ell$.
Since $u(y_0)-\phi(y_0)$ is very close to $-1$, 
i.e.\ $|u(y_0)-\phi(y_0)+1|\le \eps$ for some small $\eps>0$. 
By convexity, $u$ and $\phi$ are uniformly Lipschitz continuous near $y_0$.
Hence 

\vskip3pt

\centerline{$|u(y_0)-\phi(y_0)+1|\le C\eps$ in $B_\eps(y_0)$.}

\vskip3pt

\noo By our definition of $f$ in $B-\Om$, it follows that 

\vskip3pt

\centerline{$f(x, u)\le -C\eps^{-4n}$ in  $B_\eps (y_0)\cap(B-\Om)$. }
\vskip3pt

\noo
Since $\Om$ is Lipschitz continuous, the measure 
$|B_\eps (y_0)\cap(B-\Om)|\ge C\eps^n$.
It follows that 
the left hand side of \eqref{a4.11} is positively very large
(since $u-\ell<-\frac 14$ near $y_0$). We reach a contradiction.

\vskip5pt

\item [5.] Finally, we show that 
\beq  \label{a4.12}  
f\le C\ \ \ \text{in}\ \ B. 
\eeq
Let $\delta>0$ be a small constant. Denote $\W_{\delta}=\{u<c^*-\delta\}$. 
We have
\beqs
\int_{\W_{\delta} } U^{ij}w_{ij}(u-c^*+\delta)
 & =& -\int_{\W_{\delta}} (u-c^*+\delta)_i w_j U^{ij} \\
 & \ge &-\int_{\p\W_{\delta}} (u-c^*+\delta)_i\gamma_j w U^{ij}\\
  & =&  -\int_{\p\W_{\delta}} u_\gamma w U^{\gamma\gamma} .
  \eeqs
Denote $K_s$  the Gauss curvature of $\p\W_{\delta}$. We obtain
 \beqs
\int_{\W_{\delta} } U^{ij}w_{ij}(u-c^*+\delta)
      & =& -\int_{\p\W_{\delta}} (u_\gamma)^n w K_s\\
  & \ge & - n\omega_n \big(\sup_B |Du|\big)^n\sup_{\p\W_{\delta}} w . 
   \eeqs
Sending $\delta\to 0$ and recalling that $w=1$ on $\p B$, we obtain
$$\int_B (-f(x, u)) (u-c^*)\le C.$$
If $f(x, u)$ is very large at some point $x_0\in B-\Om$, 
then $u(x)-\phi(x)$ is close to the constant $1$ at $x_0$, 
and also in a neighborhood of $x_0$ by convexity.
Hence by the Lipschitz continuity of $\pom$, we infer as above that 
the integral on the left hand side is very large. 
The contradiction implies $|f|<C$  and Lemma \ref{L4.1} is proved.
\hfill $\square$
\eei


\vskip5pt


\begin{Theorem} \label{Thm4.1} 
Let $\Om$ be a bounded, open set in $\R^n$
with $|\pom|=0$. Let $\phi$ be a convex function defined in
a neighborhood of $\Om$, satisfying \eqref{phib}. 
Let $u$ be the maximizer of \eqref{a1.9}. 
Then there exists a sequence of smooth, 
uniformly convex solutions $u_m\in W^{4, p}(\Om)\ (\forall\ p<\infty)$ to
the equation
\beq  \label{a4.13}  
L[u]=f+\beta_m \chi(D_m)\ \ \ \ \text{in}\ \ \Om
\eeq
such that
\beq  \label{a4.14}  
|u_m-u|\to 0\ \ \ \text{uniformly in}\ \Om,
\eeq
and for any compact, proper subset $K\subset N_\phi(\Om)$,
\beq  \label{a4.15}  
K\subset  N_{u_m}(\Om)\ \ \text{when $m>>1$}, 
\eeq
where $D_m=\{x\in \Om\ |\ \text{dist}(x, \pom)<2^{-m}\}$, $\beta_m\to +\infty $ is a constant sequence, 
$\chi$ is the characteristic function, and $N_u$ is the normal mapping. 
\end{Theorem}

\noo{\it Proof}. We divide the proof into three steps. 
The first step is essentially contained in \cite{[TW2]}.

\vskip5pt

1. Let $B=B_R(0)$ as above and extend $\phi$ to $B$ by \eqref{a4.3}, 
such that $\phi$ is convex in $B$, $\phi\in C^{0,1}(\ol B)$, and 
$\phi=$const on $\p B$. 
Replacing $\phi$ by 
$\phi + \big(|x|-R+\text{$\frac 12$}\big)_+^2$,
where $(|x|-R+\frac 12)_+=\max(|x|-R+\frac 12, 0)$,
such that $\phi$ is uniformly convex in $\{x\in\R^n\ |\ R-\frac 12<|x|<R\}$.

For $j, k, m=1, 2, \cdots$, we denote
\beqs
 && H_j (t) = H(4^jt), \\ 
 && f_{m, k, j}=\left\{ \begin{array}{ll}
 f+\beta_m\chi_{D_m}\ \ & \text{in}\ \Om,\\
 0  &\text{in}\ \Om_k-\Om,\\
 H_j'(u-\phi) &\text{in}\ \ B-\Om_k,
 \end{array} \right.                               
 \eeqs
where $\Om_k=\{x\in\R^n\ |\ \text{dist}(x, \Om)<2^{-k}\}$ is the
$2^{-k}$-neighborhood of $\Om$. 

By Lemma \ref{L4.1} (with the domain $\Om$ in Lemma \ref{L4.1} replaced by $\Om_k$), 
there is a solution $u_{m,k,j}\in W^{4,p}(B)$ to 
\beqn
L[u] & =& f_{m,k,j} (\cdot, u)\ \ \ \text{in}\ \ B, \label{a4.16} \\
  u  & =&  \phi\ \ \ \ \text{on}\ \ \p B,\notag\\
  w  & =& 1 \ \ \ \text{on}\ \ \p B .\notag  
  \eeqn
By our definition of the penalty function $H_j$,  $u_{m,k,j}$ satisfies
\beq  \label{a4.17}  
|u_{m,k,j}-\phi|\le 4^{-j}\ \ \text{in}\ \ B-\Om_k. 
\eeq
Hence  $u_{m,k,j}$ converges as $j\to\infty$ to a
convex function $u_{m,k}$, which satisfies
$u_{m,k}=\phi$ in $B-\Om_k$ and $u_{m,k}$ is a solution of
\eqref{a4.13} in $\Om$.

\vskip5pt

Obviously $u_{m,k}\in \ol S[\phi, \Om_k]$.  By the proof of
Theorem 6.1 in \cite{[TW2]} 
 (here $u_{m,k}$ corresponds to $\overline u$ in the proof of
\cite[Theorem 6.1] {[TW2]}),
 $u_{m, k}$ is a maximizer of the functional
\beq  \label{a4.18}  
J_{m,k}(u)=\int_{\Om_k}\big [\det D^2 u \big]^{\frac{1}{n+2}}
          -\frac {1}{n+2}\int_{\Om_k} f_{m, k} u 
\eeq
in $\ol S[\phi, \Om_k]$, where $f_{m, k}$ is the restriction
of $f_{m,k,j}$ in $\Om_k$, which is independent of $j$.

\vskip5pt

As any function in $\ol S[\phi, \Om_k]$ is uniformly
Lipschitz, hence when $k\to\infty$, 

\vskip5pt
\centerline { $u_{m,k}\to u_m\in\ol S[\phi, \Om]$}

\noo after passing to a subsequence.
A similar proof implies that $u_m$ is the maximizer of the functional
\beq  \label{a4.19}  
J_{m}(u)=\int_{\Om} [\det D^2 u]^{\frac{1}{n+2}}
          -\frac {1}{n+2}\int_{\Om} f_{m} u 
\eeq
in the set $\ol S[\phi, \Om]$, where $f_{m}$ is the restriction
of $f_{m,k}$ in $\Om$, which is independent of $k$.

Note that by the concavity of the functional $J_m$, the maximizer of
$J_m$ in $\ol S[\phi, \Om]$ is unique. Hence as $k\to\infty$,
$u_{m,k}$ converges to $u_m$ (not just by a subsequence).

\vskip5pt

2. As functions in $\ol S[\phi, \Om]$ are uniformly
Lipschitz, so $u_m$ sub-converges as $m\to\infty$  to a convex
function $\ou\in C^{0,1}(\bom)$.
By \eqref{phib}, we have $\ou\in\ol S[\phi, \Om]$. 
We claim that $\ou$ is the maximizer of \eqref{a1.9}.

If the claim is not true, then
\beq  \label{a4.20}  
J(\ou)\le J(\tilde u)-4\eps_0
\eeq
for some $\eps_0>0$, where $\tilde u$ is the maximizer of \eqref{a1.9} and
$J$ is the functional in \eqref{a1.5}. 
Since $u_{m, k}\to u_m$ and $u_m\to \ou$,
we can choose a sequence $k_m\to\infty$ 
such that $u_{m, k_m}$ converges to $\ou$. 
Similarly to \eqref{a4.1} we introduce the function,
$$
\phi_{*m}(x)=\sup\{\L(x)\ |\ \L\ \text{is a supporting plane of $\phi$
                           at some point on $\pom_{k_m}$}\, \}. 
$$
Then 
$\phi_{*m}\in\ol S[\phi, \Om_{k_m}]$, $\phi_{*m}\le \phi$ on $\pom$,  and 
$\forall\ u\in\ol S[\phi, \Om_{k_m}]$, 
similarly to \eqref{a4.2} we have $u\ge \phi_{*m}$ in $\Om_{k_m}$.

Denote
$\mathring D_m=D_m\cup (\Om_{k_m}-\Om)$.
Let 
$$
\begin{array}{l}
\tilde u_m(x)=\sup\{a(x)\ |\ a\le \tilde u\ \text{in}\ \Om,\
 a\le \phi_{*m}\ \text{in}\ \mathring D_m, \ a\ \text{is affine}\},\\
 \ou_m (x)=\sup\{a(x)\ |\ a\le  \ou\ \text{in}\ \Om,\
 a\le \phi_{*m}\ \text{in}\ \mathring D_m, \ a\ \text{is affine}\}.
 \end{array} $$
Then $\tilde u_m, \ou_m\in \ol S[\phi, \Om_{k_m}]$ and 
\beq  \label{a4.21}  
\tilde u_m=\ou_m=\phi_{*m}\ \ \text{in}\ \ \mathring D_m. 
\eeq
By \eqref{a2.2} and the upper semicontinuity of the affine area functional,
we have
$$
\lim_{m\to\infty} \int_{\Om_{k_m}} [\det D^2 \ou_m]^{\frac{1}{n+2}}
 \le \int_\Om [\det D^2 \ou]^{\frac{1}{n+2}}. 
$$
Note that a convex function is twice differentiable a.e..
At any point where $D^2 \tilde u$ is positive definite, $u$ is strictly convex. 
Hence by our definition of $\tilde u_m$,
$\tilde u_m=\tilde u$ provided $m$ is sufficiently large. 
Hence $\det D^2 \tilde u_m\to\det D^2 \tilde u$ a.e.. 
By the upper semicontinuity of the affine area functional and the Fatou lemma, 
we therefore have
$$
\lim_{m\to\infty}  \int_{\Om_{k_m}} [\det D^2 \tilde u_m]^{\frac{1}{n+2}}
 =\int_\Om [\det D^2 \tilde u]^{\frac{1}{n+2}}. 
$$
It follows by \eqref{a4.20}
\beq  \label{a4.22}  
J_{m, k_m} (\ou_m)\le J_{m, k_m} (\tilde u_m)-3\eps_0 
\eeq
when $m$ is sufficiently large. Note that in \eqref{a4.22}, 
the integration involving $\beta_m$ is irrelevant, as
$\tilde u_m=\phi_{*m} = \ou_m$ in $D_m$.  

Next observe that when $m$ is sufficiently large, 
\beq  \label{a4.23}  
J_{m, k_m}(u_{m, k_m})\le J_{m, k_m}(\ou_m)+\eps_0 .
\eeq
To see this, recall that by our notation,
$$J_{m,k}(u_{m, k})=
\int_{\Om_k} [\det D^2 u_{m, k}]^{\frac{1}{n+2}}
          -\frac {1}{n+2}\int_{\Om_k} f u_{m, k}
          -\frac {1}{n+2}\int_{D_m} \beta_m u_{m, k}.$$
The first two terms on the right hand side give no trouble, 
but for the last term, we have by definition, 
$u_{m, k_m}\ge \ou_m$ in $D_{m}$.
Hence \eqref{a4.23} holds and we obtain from \eqref{a4.22} that 
\beq  \label{a4.24}  
J_{m, k_m}(u_{m, k_m})\le
J_{m, k_m} (\tilde u_m)-2\eps_0 .                               
\eeq
But this is a contradiction as $u_{m, k_m}$ is a maximizer of $J_{m, k_m}$ in 
$\ol S[\phi, \Om_{k_m}]$ and $\tilde u_m\in \ol S[\phi, \Om_{k_m}]$. 
Therefore $\ou$ is a maximizer of \eqref{a1.9}.

\vskip5pt

3. {\it Claim:}\  
$u_m$ satisfies \eqref{a4.15} when $\beta=\beta_m$ is sufficiently large.

 Indeed, by convexity it suffices to prove that
for any fixed $m$,
\beq  \label{a4.25}  
\lim_{\beta\to\infty} u_m(x)\le \phi_*(x)\ \ { \forall\ x\in D_m,}
\eeq
where $\phi_*$ is given in \eqref{a4.1}.
Hence in \eqref{a4.13}, we allow that $\beta_m>>2^m$.

For any given $m$, the function $u_m$  
is a maximizer to the concave
functional $J_m$ in \eqref{a4.19} in the set $\ol S[\phi, \Om]$. If \eqref{a4.25}
is not true, there is a point $x_0\in D_m$ such that $u_m(x_0)\ge
\phi_*(x_0)+\eps_0$ for some constant $\eps_0>0$. 
By convexity, $u_m$
and $\phi_*$ are uniformly Lipschitz continuous. Hence $u_m(x)\ge
\phi_*(x)+\frac 12\eps_0$ in $B_r(x_0)$, where $r= C\eps_0$ for some
small constant $C>0$. 
Let
$$u_{*m}(x)=\sup\{a(x)\ |\ a\le u_m\ \text{in}\ \Om,\
 a\le \phi_{*}\ \text{in}\ D_m,  \ a\ \text{is affine}\}.$$
Then $u_{*m}\in \ol S[\phi,\Om]$ and  $u_{*m}$ satisfies $u_{*m}=\phi_*$ in $D_m$ 
and $u_m\ge u_{*m}\ge \phi_*$ in $\Om$. 
When $\beta$ is sufficiently large, the
integration $-\int_{D_m}\beta (u_{*m}-u_m)$ becomes very large, which
dominates the (uniformly bounded) first two terms on the right hand side of the functional 
$$J_m(u_m)=\int_{\Om} [\det D^2 u_m]^{\frac{1}{n+2}}
          -\frac {1}{n+2}\int_{\Om} f u_m -\beta\int_{D_m} u_m.$$
Hence $u_m$ cannot be a maximizer of $J_{m}$ in $\ol S[\phi, B]$ and we reach a contradiction.

We have thus obtained the sequence $u_m$ which satisfies \eqref{a4.13} and \eqref{a4.15},
and $u_m$ is the limit of the smooth functions $u_{m, k, j}$ as $j\to\infty$ and $k\to\infty$.

Note that the proof above does not directly imply that the function
$u_m$ (obtained in Step 1) is smooth.
However, since $u_m$ is the limit of the smooth functions $u_{m, k, j}$ as $j\to\infty$ and $k\to\infty$,
we can select subsequences $k_m\to\infty$ and $j_m\to\infty$, such that
the function $u_{m, k_m, j_m}$ (still denoted as $u_m$) satisfies both requirements \eqref{a4.13} and \eqref{a4.15}.
This completes the proof.
\hfill $\square$

\vskip10pt

From the proof, one sees that the main ideas are the following:
\begin{itemize}
\item [(I1)] As $j\to\infty$, the penalty function $H_j$ in \eqref{a4.16} pushes the solution $u_{m,k,j}$ 
as close as possible to $\phi$ in $\Om^\delta-\Om$.

\item [(I2)] As $m\to\infty$, the second penalty $\beta_m\to\infty$ and it pushes the solution $u_{m, k}$ 
as close as possible to $\phi_*$ in $D_m$.

\item [(I3)] From (I1), (I2) and \eqref{phib}, we obtain \eqref{a4.15}.
\end{itemize}

The approximation \eqref{a4.14} was proved in \cite{[TW2]}, 
but \eqref{a4.15} is new.
We point out that \eqref{a4.15} is critical for this paper, 
as it enables us to apply the a priori estimate \eqref{a2.15} to $u_m^*$
(the Legendre transform of $u_m$).
See Remark 4.1 (ii) below for details.

Denote $U_m=N_{u_m}(\Om)$. Since $u_m\in \ol S[\phi, \Om]$, 
we have $U_m\subset \Om^*=:N_\phi(\Om)$.
\eqref{a4.15} tells that for any compact, proper subset $K\subset N_\phi(\Om)$, we have
$K\subset U_m$ if $m>>1$.
From Step 3 in the proof of Theorem \ref{Thm4.1}, we also have the following property

\begin{Corollary}\label{Coro4.1}
For any $x\in\p U_m$, 
\beq\label{a4.30}
|u^*_m(x)-\phi^*(x)|\to 0\ \ \text{as}\ m\to\infty ,
\eeq
where $u^*_m$ and $\phi^*$ are the Legendre transforms of $u_m$ and $\phi$, respectively.
\end{Corollary}

Indeed,  since $u_m\in \ol S[\phi, \Om]$, we have $u^*_m(x)\le \phi^*(x)+\eps$ near $\p U_m$,
with $\eps\to 0$ as $x$ approaches $\p U_m$.
On the other hand, we may choose $\beta_m$ in \eqref{a4.25} so large such that
$\max\{ (u_m-\phi) (x), 0\}$ is so small (compare with the width of the strip $D_m$),
and hence we have \eqref{a4.30}.


\vskip5pt

\noo{\bf Remarks 4.1}.
\newline
(i) The uniqueness of the maximizer $u$ of \eqref{a1.9} is crucial in the proof of Theorem \ref{Thm4.1}.
It ensures that the sequence $u_m$ converges to the unique maximizer.
When $\th\in (0, \frac 1n]$, the uniqueness follows from the
concavity of the functional $J$. 
\\ [5pt]
(ii) Let $u^*$ be the Legendre transform of $u$.
\eqref{a4.15} and \eqref{a4.30} enable us to apply the a priori estimates to $u^*$,
which is precisely what is needed in this paper. 
To see this,  consider a sequence of smooth solutions $u_m\to u$
to the affine mean curvature equation. 
Let $u_m^*$ and $\phi^*$ be the Legendre transform of $u_m$ and
$\phi$, respectively. By   \eqref{a2.9}, $u_m^*$ is a smooth
function in $\Om_\delta^*$, satisfying equation
\beq  \label{a4.26}  
L^*[u^*]= - {\Small\text{$\frac{1}{n+2}$}} \big[f(Du^*)+\beta_m\chi_{D_m}\big]\,\det D^2 u^*
 \ \ \ \ \text{in}\ \ \Om^*_\delta, 
 \eeq
with $\delta\to 0$ as $m\to\infty$, where
$\Om^*_\delta=\{x\in\Om^*\ |\ \text{dist}(x, \pom^*)>\delta\}$. 
By \eqref{a4.15} and \eqref{a4.30}, 
we can apply the a priori estimate \eqref{a2.15} to $u^*_m$ in an appropriate sub-level set, 
and the upper bound $C$ in \eqref{a2.15} is independent of $\beta_m$.  
Passing to the limit then yields the a priori estimate \eqref{a2.15} for $u^*$.

\vskip20pt
 
 \section{Regularity for the first boundary value problem}
 
In this section we prove Theorem \ref{Thm2.3}.
For convenience we restate the theorem here.

\begin{Theorem}\label{Thm5.1} 
Let $\Om$ be a smooth, bounded domain in $\R^n$
and $\phi$ a smooth, uniformly convex function defined in $\bom$.
Then if $f\in L^\infty(\Om)$, 
the maximizer of \eqref{a1.9} is uniformly convex and lies in the Sobolev space $W^{4,p}_{loc}(\Om)$
for any $p>1$; if $f\in C^\alpha(\Om)$ for some $\alpha\in
(0, 1)$, the maximizer is  $C^{4,\alpha}$ smooth
in $\Om$.
\end{Theorem}

Let $u$ be a maximizer of \eqref{a1.9}. To prove the regularity of $u$, it
suffices to prove the strict convexity of $u$, as was pointed out in the
introduction. This section is devoted to the proof of the strict convexity.  

We first prove in \S5.1 that the graph of $u$, $\M_u$, 
contains no line segment with both endpoints on the boundary. 
In \S5.2 we rule out the possibility
that $\M_u$ contains a line segment with one endpoint on the
boundary. Finally in \S5.3 we prove that $\M_u$ contains no line
segment of which both endpoints are interior points of $\M_u$.

\vskip10pt

\subsection {Strict convexity I}\ 

\begin{Lemma} \label{L5.1} 
Suppose that $\Om$ is a smooth and bounded domain in $\R^n$, 
$\phi$ is a smooth and uniformly convex function in $\bom$. 
Let $u$ be a maximizer of \eqref{a1.9}. 
Then $\M_u$ contains no line segment with both endpoints on $\p\M_u$.
\end{Lemma}

\noo{\it Proof}. Suppose $\ell$ is a line segment in the graph
$\M_u$ with both endpoints on boundary $\p\M_u$. By subtracting a
linear function we may assume that
\beq \label{a5.1}
u\ge 0
\eeq
and the line segment $\ell$ lies in the plane $\L:=\{x_{n+1}=0\}$.
We may furthermore assume that
\beq  \label{a5.2}
\ell=\{x\in\L\ |\   x'=0, -1\le x_n\le 1\},
\eeq
with the endpoints $\pm e_n\in\pom$, where $x'=(x_1,\cdots,
x_{n-1})$ and $e_n=(0,\cdots, 0, 1)\in\R^n$ is the unit vector in the
$x_n$-axis.

For sufficiently small $\delta >0$, we now extend $\phi$ 
to the neighborhood $\Om^\delta =\{x\in\R^n\ |\ \text{dist}(x, \Om)<\delta\}$,
such that $\phi$ is smooth and uniformly convex.
Since $\phi=u=0$ at $\pm e_n$, the
line segment $\ell$ must be transversal to $\M_\phi$ at $\pm e_n$,
namely the inner product $\lan e_n, \gamma \ran\ne 0$, where
$\gamma$ denotes the normal of $\M_\phi$ either at $e_n$ or at $-e_n$. Hence by
$\phi=u\ge 0$ on $\pom$ and the smoothness of $\phi$, there exists a
constant $C$ such that
\beq   \label{a5.3}
u(x)=\phi(x)\le {\Small\text{$\frac 12$}} C |x'|^2
                          \ \ \forall\  x\in \pom. 
\eeq
By the convexity of $u$, it then follows that
\beq   \label{a5.4}
u(x)\le {\Small\text{$\frac 12$}} C |x'|^2 \ \forall\ x\in\Om. 
\eeq

Since both endpoints $\pm e_n$ are boundary points of $\M_\phi$, from
the uniform convexity of $\phi$, we have $0\not\in D\phi(\pom)$. Hence 
$0\in \Om^*:=D\phi(\Om)$.

Let $u^*$ be the Legendre transform of $u$, defined in \eqref{a2.3}, 
which is a convex function defined in $\Om^*$. 
Since $u(0)=0$ and $u\ge 0$, we have
$u^*(0)=0$. Since the line segment $\ell$, given by \eqref{a5.2}, lies in
the graph $\M_u$, one easily verifies that, (see (i)-(iii) after
\eqref{a2.3}),
\beq   \label{a5.5}
u^*(0, y_n)\ge |y_n|.
\eeq
From \eqref{a5.4}, we also have
\beq   \label{a5.6}
u^*(y)\ge {\Small\text{$\frac{1}{2C}$}} |y'|^2.
\eeq
Therefore $\det D^2 u^*$ is not bounded from above near the origin 
(see Remark 2.2).

But on the other hand, by the a priori estimates (Lemma \ref{L2.3}), 
$\det D^2 u^*$ must be bounded. Indeed, by Theorem \ref{Thm4.1},
there exists a smooth approximating sequence $u_m\to u$. 
By equation \eqref{a2.9} or \eqref{a4.26}, $u^*_m$ satisfies the
equation
\beq   \label{a5.7}
U^{*ij} w^*_{ij}
          = - {\Small\text{$\frac {1}{n+2}$}} f_m(Du^*)\,\det D^2 u^* 
\eeq
with $f_m=f+\beta_m\chi_{D_m}$. The assertions \eqref{a4.14} and \eqref{a4.15}
imply that $u^*_m$ is defined in $\Om^*_{\delta_m}$ with
$\delta_m\to 0$ as $m\to\infty$, where $\Om^*_\delta=\{x\in\Om^*\ |\
\text{dist}(x, \pom^*)>\delta\}$.

By the growth estimates \eqref{a5.5} and \eqref{a5.6},  the set $\{u^*<h\}$ is
strictly contained in $\Om^*$ provided $h>0$ is a small constant.
Hence we can apply Lemma \ref{L2.3} (in the case $\th=\frac {n+1}{n+2}$) to
get the estimate
$$\det D^2 u^*_m\le C_1,$$
where the constant $C_1$ is independent of
$\beta_m$. Sending $m\to\infty$, by the weak convergence of the
Monge-Amp\`ere operator, we obtain $\det D^2 u^*\le C_1$ near the
origin. This is in contradiction with the assertion that $\det D^2
u^*$ is not bounded from above near the origin. $\square$

\vskip10pt

\subsection{Strict convexity I\!I}

In this section we rule out the possibility that $\M_u$ contains a
line segment with one endpoint on the boundary.

\begin{Lemma}\label{L5.2} 
Suppose $\Om$ is a smooth and bounded domain
in $\R^n$, and $\phi$ is a smooth, uniformly convex function in $\bom$. Let $u$
be a maximizer of \eqref{a1.9}. Then $\M_u$ contains no line segment with
an endpoint on $\p\M_u$.
\end{Lemma}

\noo{\it Proof}. Assume that $\ell$ is a line segment in the graph
$\M_u$ with an endpoint on the boundary $\p\M_u$.

First we choose the coordinates such that the endpoint is
the origin, $e_n=(0, 0, \cdots, 1)$ is the inner normal of $\pom$ at
the origin, and locally $\pom$ can be represented as
\beq \label{a5.8}
x_n=\eta(x')
\eeq
with $\eta(0)=0$ and $D\eta(0)=0$, where $x'=(x_1, \cdots, x_{n-1})$.
By the smoothness and uniform convexity of $\phi$, $\ell$ cannot lie in the $x'$-plane.

By the affine invariance, we can make a linear transform such that $\ell$ is given by
\beq \label{a5.9}
\ell=\{te_n\ |\ \ 0\le t\le 1\}. 
\eeq
Therefore $0$ and $e_n$ are the two endpoints of $\ell$.

Let $\L$ be a tangent plane of $\M_u$ containing $\ell$. 
By  Lemma \ref{L5.1}, $\L\cap\M_u$ contains only one boundary point.  
By subtracting a linear function we may
assume that  $\L=\{x_{n+1}=0\}$ and $u\ge 0$.
Then the origin is contained in $Du(x)$ for any point $x\in\ell$.

\vskip5pt
 
\noo $\underline{\text{Claim 1}}$: The origin is a boundary point of
$\Om^*:=D\phi(\Om)$.

Indeed, we have $u=0$ on $\ell$ and $u\ge 0$. By the smoothness of
$\phi$ and the convexity of $u$ we have
$$u(x)\le C |x'|^2$$
on the boundary. Hence by convexity, we have
\beq \label{a5.10}
u(x', x_n)\le C |x'|^2\ \ \forall\ x_n\in (0, 1-\delta),
\eeq
where $C$ depends on $\delta$. To prove \eqref{a5.10}, we consider the line
segment $l=\{(1-t)p_0+tp_1\ |\ t\in (0, 1)\}$, where $p_0=(x',
\eta(x'))$ is a boundary point and $p_1=(0, \cdots, 0, 1)$. Then
$u(p_1)=0$ and by convexity, $u(p_t)\le (1-t)u(p_0)\le C|x'|^2$). 

By the monotonicity of the Legendre transform (see property (i) in \S2.2), it follows
that
\beq \label{a5.11}
u^*(y)\ge {\Small\text{$\frac 1{2C}$}}  |y'|^2. 
\eeq
Moreover, since $\ell=\{te_n\ |\ \ 0\le t\le 1\}$, we have
\beq \label{a5.12}
u^*(y)\ge \max\{0, y_n\}.
\eeq
Subtracting $\frac 12 y_n$ from $u^*$, we obtain
\beq  \label{a5.13}
u^*(y)\ge {\Small\text{$\frac 12$}}  |y_n|. 
\eeq

If the origin is an interior point of $\Om^*$, then $u^*(0)=0$ and
by \eqref{a5.11}, \eqref{a5.13}, we have $u^*>0$ on $\pom^*$. Hence by Lemma \ref{L2.3}
and the approximation in Remark 4.1(ii), we have
\beq  \label{a5.14}
\det D^2 u^*\le C.
\eeq
But this is impossible due to \eqref{a5.11} and \eqref{a5.12}; (see Remark 2.2). The
claim is thus verified.

From Claim 1, we have $\phi(0)=0$ and $D\phi(0)=0$. By the Taylor
expansion, we have
\beq  \label{a5.15}
\phi(x)=\frac 12 {\sum}_{i, j=1}^n a_{ij} x_ix_j+ O(|x|^3). 
\eeq
By a rotation of the axes $x'$, we may assume that
$\{a_{ij}\}_{i,j=1}^{n-1}$ is diagonal. 
The affine invariance also allows us to assume that $\{a_{ij}\}_{i,j=1}^{n-1}$ is the unit matrix.

Recall that $\eta(0)=0$ and $D\eta(0)=0$ in \eqref{a5.8}, 
and near the origin, $\pom^*=D\phi(\pom)$.
Hence
$\pom^*$ can be represented as
\beq  \label{a5.16}
y_n={\sum}_{j=1}^{n-1}b_jy_j+O(|y'|^2),
\eeq
where $b_j=\frac{a_{nj}}{a_{jj}}=a_{nj}$ ($j=1, \cdots, n-1$). 
Hence $\{y_n=\sum b_jy_j\}$ is a tangent plane of $\Om^*$ at the origin.

\vskip5pt

\noo $\underline{\text{Claim 2}}$: $b_j=0$ for all $j=1, \cdots,
n-1$. Hence $\{a_{ij}\}$ is diagonal and $e_n$ is an inner normal to $\pom^*$ at $0$.

Indeed, if not, let us assume that $b_1>0$. For any given
$\delta>0$, denote 
\vskip5pt
\centerline{ $p'_\delta=(\delta, 0, \cdots, 0)\in\R^{n-1}$,\  $p_\delta=(p'_\delta, t_\delta)\in\pom$,}

\noo
where $t_\delta=\eta(p'_\delta)$ and $\eta$ is the function in \eqref{a5.8}.
Let
$\cR_\delta=\{(p'_\delta, t)\ |\  t>t_\delta\}$
be a ray starting at $p_\delta$, parallel to the $x_n$-axis. Consider the function
$u$ on the ray $\cR_\delta$. By convexity,
$$u(p'_\delta, t)\ge u(p_\delta)+u_{x_n}(p_\delta) (t-t_\delta). $$
Since $u\in \ol S[\phi, \Om]$, we have $Du(\Om)\subset
D\phi(\Om)=\Om^*$. Hence when $\delta>0$ is small,
$$
u_{x_n}(p_\delta)\ge \phi_{x_n}(p_\delta)  \ge {\Small\text{$ \frac \delta 2$}} b_1. 
$$
Noting that $u\ge 0$,  we have
$$u(p'_\delta, t)\ge u(p_\delta)+C\delta(t-t_\delta)
                \ge C\delta (t-t_\delta)$$
Fixing a $\tilde t\in (t_0, 1]$, we obtain
$$u(\delta, 0, \cdots, 0, \tilde t)\ge C_1\delta\ \ \ \forall\ \delta>0, $$
where $C_1=C(\tilde t-t_0)$. 
Since $u=0$ on the line segment $\ell=\{te_n\ |\ t\in [0,1]\}$ and $u\ge 0$, 
it implies that $x_{n+1}=C_1x_1$ is a supporting plane of
$u$ at any point on the line segment $\{te_n\ |\ t\in [0,\tilde t)\}$,
including the endpoint $0$. It follows that on the boundary $\pom$,
we have
$$u=\phi\ge C_1 x_1,$$
which is a contradiction with \eqref{a5.10}. The claim is thus verified.

\vskip5pt

\noo $\underline{\text{Claim 3}}$: For any $x'$ near $0$ and any
$x_n\in (0, 1]$ such that $x=(x', x_n)\in\Om$, we have
\beq  \label{a5.19}
u(x', x_n)\ge C_*|x'|^2 
\eeq
for some positive constant $C_*$.

Indeed, for any $x'$ near $0$, let $x_b=(x', \eta(x'))\in\pom$ be
a boundary point, and $\cR_{x'}=\{(x', t)\ |\ t>\eta(x')\}$ be a ray
parallel to the $x_n$-axis. We consider the function $u$ on the line
$\cR_{x'}$. By convexity we have
$$u(x', t)\ge u(x_b)+u_{x_n}(x_b)(t-\eta(x')). $$
By Claim 1, a boundary point $(y', y_n) \in \pom^*$ near the origin satisfies
\beq  \label{a5.17}
y_n\ge -C_2|y'|^2.  
\eeq
By the uniform convexity  and smoothness of $\phi$, we also have 
\beq  \label{a5.18}
\phi(x)\ge C_3|x'|^2 \ \ \forall\ \ x\in\pom, 
\eeq
where $C_2, C_3$ are positive constants.

By \eqref{a5.17} and since $u\in\ol S[\phi,\Om]$, we have
$$u_{x_n}(x_b)\ge \phi_{x_n}(x_b)\ge -C_4|x'|^2.$$
Hence by \eqref{a5.18},
\beq  \label{a5.20}
u(x', t)\ge
   C_3|x'|^2-C_4|x'|^2(t-\eta(x'))\ge  {\Small\text{$ \frac 12 $}} C_3|x'|^2 
   \eeq
when $t-\eta(x')$ is small. To obtain \eqref{a5.19}, we consider the
restriction of $u$ on a line segment with endpoints $0$ and $x=(x',
x_n)\in\Om$. By  the convexity of $u$, we have
\beq  \label{a5.21}
u(x)\ge {\Small\text{$\frac 1t$}} u(tx)\ge Ct |x'|^2,
\eeq
where the second inequality is because when $t>0$ is small, by
\eqref{a5.20} we have $u(tx)\ge C |tx'|^2$. We have thus proved \eqref{a5.19}.

From \eqref{a5.19}, we have $\ell=\{x\in\Om\ |\ u(x)=0\}$, namely $u(x)>0$
for any point not on the line segment $\ell$.

 \vskip5pt

\noo $\underline{\text{Claim 4}}$: For any $x'$ near $0$ and any
$x_n\in (0, 1)$ such that $x=(x', x_n)\in\Om$, we have
\beq  \label{a5.22}
u(x', x_n)\le C^*|x'|^2 
\eeq
for some positive constant $C^*$ (here $C^*$ depends on $x_n$, and $\eta(x')<x_n<1$.).

Indeed, let $\cR_x$ be a line segment connecting a boundary point
$x_b=(x'_b, \eta(x'_b))\in\pom$ to the point  $e_n=(0, \cdots, 0, 1)$,
which is an end point of the line segment $\ell$. 
Given a point $x=(x', x_n)\in \cR_x$, by the boundary
condition $u=\phi$ and noting that $\L=\{x_{n+1}=0\}$ is a supporting
plane of $u$ on the line segment $\ell$, we have $u(x_b)\le
C|x_b'|^2$. Inequality \eqref{a5.22} then follows from the convexity of $u$ on the segment
$\cR_x$.
(Note that $x_n<1$ in \eqref{a5.22}, and the constant $C^*$ depends on $x_n$.
The constant $C^*$ is not uniformly bounded when $x_n$ approaches to 1).

\vskip10pt

With \eqref{a5.19} and \eqref{a5.22}, we can prove Lemma \ref{L5.2} by a simple blow-up
argument. For any small $\eps>0$, let $\omega_\eps=\{x\in\Om\ |\
u(x)<\eps\}$. Let
$$y'= x'/\sqrt\eps,\ \ \ y_n  =x_n,$$
and let
$$u_\eps(y)=u(x)/\eps.$$
At $x_n=\frac 12$, by \eqref{a5.19} and \eqref{a5.22}, we have
\beq  \label{a5.23}
C_*|y'|^2\le u_\eps(y', \frac 12)\le C^* |y'|^2.
\eeq
Therefore by the proof of Claims 3 and 4, we see that $u_\eps$ also
satisfies \eqref{a5.19} and \eqref{a5.22} with the constant $C_*$ and $C^*$
independent of $\eps$.

Denote by $\M_\eps$ the graph of $u_\eps$. 
Then as $\eps\to 0$, by estimate \eqref{a5.23} 
we see that $\M_\eps$ sub-converges to a convex hypersurface $\tilde\M$. 
Since $u(x)>0\ \forall\ x\in\Om\backslash\ell$, we see that
\beq   \label{a5.24}
\tilde M\subset \{0\le y_n\le 1\}\cap\{y_{n+1}\ge 0\}.
\eeq
Apparently $\ell\subset \tilde \M$. The blow-up also implies that the ray
$$ \cR=\{e_n+te_{n+1}\ |\ t\ge 0\}\subset \tilde\M ,$$ 
where $e_n$ and $e_{n+1}$
are the unit vectors on the $y_n$ and the $y_{n+1}$ axes. Now we
choose a new coordinate system $z=(z_1, \cdots, z_{n+1})$ such that
\beqn 
 & z'=y', \no\\
 & z_n=\frac{\sqrt 2}{2}[(y_n-1)-y_{n+1}],    \label{a5.25} \\
 & z_{n+1}=\frac{\sqrt 2}{2}[(y_n-1)+y_{n+1}].\no
\eeqn
namely we first translate the coordinates such that the point $e_n$
is the origin, then rotate the $y_ny_{n+1}$-plane by 45 degrees. 
In the $z$-coordinates, the hypersurface $\tilde\M$ is the 
graph of a convex function $v$ which satisfies (near the origin)
\beqn
 &&  v(z', z_n)\ge C|z'|^2,     \label{a5.26} \\
 &&  v(z', z_n)\ge |z_n|,     \label{a5.27} 
 \eeqn
where \eqref{a5.26} follows from \eqref{a5.23} and \eqref{a5.27} follows directly from \eqref{a5.24}.

By the upper semi-continuity of the affine area functional, $\tilde
\M$ is affine maximal. But by Remark 2.2, an affine maximal
hypersurface cannot satisfy simultaneously \eqref{a5.26} and \eqref{a5.27}, so we
reach a contradiction. The contradiction implies there is no line
segment in the graph of $u$ for which one end point is a boundary
point. \hfil$\square$


\vskip10pt

\subsection {Strict convexity I\!I\!I}

With  Lemmas \ref{L5.1} and \ref{L5.2}, 
we can finally  prove that the maximizer $u$ is strictly convex.

\begin{Lemma} \label{L5.3} 
Suppose that $\Om$ is a smooth and bounded domain in $\R^n$, 
and $\phi$ is a smooth, uniformly convex function in $\bom$. Let $u$
be a maximizer of \eqref{a1.9}. Then $\M_u$ is strictly convex.
\end{Lemma}

\noo{\it Proof}. With Lemmas \ref{L5.1} and \ref{L5.2}, 
it remains to rule out line segments in the graph $\M_u$,
of which both endpoints are interior points of $\M_u$.

First we note that for any $\delta>0$, by Lemmas \ref{L5.1} and \ref{L5.2}
and a compactness argument, there exists a constant $\sigma>0$, depending
on $\delta, n, \Om$ and the boundary condition $\phi$, but
independent of $u$, such that for any point $x_0\in\Om_\delta$,
\beq \label{a5.28} 
u(x)-u(x_0)-a \cdot (x-x_0)\ge \sigma, \ \ \forall\ x\in\pom, 
\eeq
where $u(x_0)+a \cdot (x-x_0)$ is a supporting plane of $u$ at $x_0$.
By Lemma \ref{L2.3}, we obtain the estimate
\beq \label{a5.29} 
\det D^2 u\le C\ \ \text{in}\ \ \Om_\delta, 
\eeq
where $C=C_\delta$ depends on $\delta$.
Recall that $u\in \ol S[\phi, \Om]$ has bounded gradient. 

For any point $x_0\in\Om_\delta$, by subtracting a linear function
we assume that $u(x_0)=0$ and $u(x)\ge 0$ in $\Om$. Let $u^*$ be the
Legendre transform of $u$. Then $u^*(0)=0$, $u^*\ge 0$ in $\Om^*$,
and $Du^*(\Om^*)\subset\Om$, namely $|Du^*|$ is bounded in $\Om^*$.
By \eqref{a5.28}  there is a positive constant $\sigma^*>0$ depending on
$\delta, \sigma, n, \Om$ and $\phi$, such that $u^*\ge\sigma^*$ on
$\pom^*$. Hence for any $a<\sigma^*$, if $u^*$ is smooth, applying the a priori estimate 
\eqref{a2.15},  we obtain 
\beq\label{a5.30}
\det D^2 u^*\le C\ \ \text{in}\ \ D_a:=\{x\in\Om^*\ |\ u^*(y)<a\}.
\eeq
If $u^*$ is not
smooth, then by Theorem \ref{Thm4.1}, $u^*$ can be approximated by smooth solutions.
Hence  we can also apply  \eqref{a2.15} to obtain \eqref{a5.30}.

On the other hand, by \eqref{a2.6} we have $\det D^2 u^*=[\det D^2 u]^{-1}$. 
Hence  by \eqref{a5.29} we have 
\beq\label{a5.31}
\det D^2 u^*\ge C^*_\delta\ \ \text{in}\ \  \Om^*_\delta ,
\eeq
where $C^*_\delta$ is a constant depending on $\delta$ and
$\Om^*_\delta=\{x\in\Om^*\ |\ \text{dist}(x, \pom^*)>\delta\}$. 
Note that by approximation and the weak convergence of the Monge-Amp\`ere measure, 
\eqref{a5.31} holds for non-smooth solutions.

By \eqref{a5.30} and \eqref{a5.31},  and by Lemma \ref{L2.2},
we conclude that $u^*$
is strictly convex and $C^1$ smooth in $D_a$. Hence the mapping
$y\to Du^*(y)$ is 1-1. Recall that $(u^*)^*=u$ (the Legendre
transform of $u^*$ is $u$). Hence $u$ satisfies \eqref{a2.11} in $\Om_a$
for $a>0$ small. Since $u$ is affine maximal, the singular part of
$\mu_u$ vanishes \cite{[TW2]}. Hence \eqref{a2.11} holds and so $u$ is strictly
convex near the origin. Hence the graph of $u$ does not contain any
line segment.  \hfill $\square$

\vskip10pt

\noo{\bf Proof of Theorem \ref{Thm5.1}}. 
From Lemmas \ref{L5.1}, \ref{L5.2} and \ref{L5.3},  the
maximizer of \eqref{a1.9} is strictly convex. By Theorem \ref{Thm4.1}, the maximizer
can be approximated by smooth solutions. Hence by the a priori
estimates in Proposition \ref{P2.1}, the maximizer is smooth and uniformly
convex in $\Om$. Theorem \ref{Thm5.1} is thus proved. \hfill $\square$

 \section{Regularity for the affine Plateau problem}

In this section, we prove Theorem \ref{Thm2.4}, that is
 
\begin{Theorem} \label{Thm6.1} 
Let $\M_0$ be a bounded, connected, smooth,
and locally uniformly convex hypersurface in $\R^{n+1}$, with $C^\infty$ smooth
boundary $\Gamma =\p\M_0$. Suppose the Gauss image of $\M_0$ does not
contain any hemisphere. Then the maximizer of \eqref{a1.10} is strictly
convex and $C^\infty$ smooth.
\end{Theorem}

A {\it locally convex} hypersurface $\M$ is an immersion of an oriented
manifold in $\R^{n+1}$ such that near any point $p\in\M$, $\M$
can be represented as the graph of a convex function, and the
orientation of convexity agrees with the orientation of the
manifold. A locally convex hypersurface $\M$ is {\it convex} if at
any point $p\in\M$, $\M$ lies on one side of $\L_p$, where $\L_p$ is
a supporting plane of $\M$ at $p$.

We first formulate the following fundamental lemma for locally convex
hypersurfaces.

\begin{Lemma} \cite{[TW4], [VH]} \label{L6.1} 
Let $\M$ be a compact, locally convex
hypersurface (not necessarily smooth). Suppose that the boundary
$\p\M$ lies on a hyperplane. Then $\M$ is convex.
\end{Lemma}

Let $\M_0$ be as in Theorem \ref{Thm6.1}. By Lemma \ref{L6.1} we have:
$$\begin{array}{l}
 \text{(a1) 
   there exists $R>0$ such that $\M\subset B_R(0)$ $\forall\ M\in\ol S[\M_0]$;}\\
 \text{(a2) 
   there exists $r>0$ such that for any $\M\in\ol S[\M_0]$ and any point $p\in\M$,}\\
   \text{ \ \ \ \ \ the component of $B_r(0)\cap \M$ containing $p$ is a
convex graph.}
\end{array} $$
We  refer the reader to \cite{[TW2], [TW4]} for discussions on locally convex hypersurfaces with boundary.

For a locally convex hypersurface $\M$,  recall that the affine area
functional is given by $A(\M)=\int_\M K^{1/(n+2)}$ when $\M$ is
smooth. Similarly to \eqref{a2.1} we extend it to nonsmooth, 
locally convex hypersurfaces by letting $K$ be the Gauss curvature of $\M$ 
at twice differentiable points, and $K=0$ elsewhere. 
Then $A$ is also upper semicontinuous. Therefore we have {\cite{[TW2]}}

\begin{Lemma}  \label{L6.2}  
Let $\M_0$ be as in Theorem \ref{Thm6.1}.  Then there is a locally convex
maximizer $\M$ to
$\sup_{\M'\in \ol S[\M_0]} A(\M') . $
\end{Lemma}

The condition that the image of the Gauss mapping of $\M_0$ does not
cover any hemi-sphere is necessary \cite{[TW2]}, otherwise the supremum
$\sup_{\M\in \ol S[\M_0]} A(\M) $ is unbounded.

\vskip10pt

Before proving   Theorem \ref{Thm6.1}, we introduce some notation.
As in \cite{[TW2]}, we extend $\M_0$ to a smooth, locally uniformly convex
hypersurface $\wtt \M_0$ such that any point $p\in \p\M_0$ has a positive geodesic 
distance $\delta_0>0$ to $\p \wtt\M_0$. For any $\M\in \ol S[\M_0]$, we also extend
$\M$ to a locally convex hypersurface $\wtt\M$ such that $\wtt\M-\M=
\wtt\M_0-\M_0$.

Let $\M$ be a maximizer in Lemma \ref{L6.2}. For any point $p\in\M$, let $\L$
be a supporting plane of $\M$ at $p$. By a proper choice of the
coordinates, we assume that $p=0$, $\L=\{x_{n+1}=0\}$, and $-e_{n+1}=-(0,
\cdots, 0, 1)$ is a unit normal of $\M$ at $p$ (pointing to the convex side).
Move the hyperplane $\L$ upwards by a small distance $\eps\in (0,
\eps^*)$, where $\eps^*>0$ depends only on the extended part
$\wtt\M_0-\M_0$, such that $\L$ does not touch the boundary
$\p\wtt\M$. Then by Lemma \ref{L6.1}, $\L$ cuts off a convex cap $\M_\eps$
from $\wtt\M$, and $\p\M_\eps$ is an $(n-1)$-dimensional
convex hypersurface in the hyperplane $\{x_{n+1}=\eps\}$.

\vskip10pt

Let $Q_\eps$ denote the convex body enclosed by the cap $\M_\eps$ 
and the hyperplane $\L=\{x_{n+1}=\eps\}$. 
We make a linear transform of the coordinates such that $Q_\eps$ is {\it normalized}
(see normalization before Lemma \ref{L2.4}).
We may simply say that $\M_\eps$ is normalized.


Let $D_\eps$ be the projection of $\M_\eps\cup\p\M_\eps$ on the plane
$\L=\{x_{n+1}=0\}$. Then $D_\eps$ is a closed convex set and $\M_\eps$ is
contained in the cylinder $D_\eps\times [0, 1]$. 
We can define a convex function $u_\eps$ in $D_\eps$ as follows. 
For any interior point $x\in D_\eps$, let
\beq  \label{a6.4}
u_\eps (x)=h , 
\eeq
where $h$ is chosen such that $(x, h)\in\M_\eps$ 
and the normal of $\M_\eps$ at $(x, h)$ is in the south hemisphere. 
Since $\M_\eps$ is convex, this $h$ is uniquely determined.
By definition, the graph of $u_\eps$ is a part of $\M_\eps$.
We can extend $u_\eps$ to $\p D_\eps$
by letting $u_\eps(x)={\underline\lim}_{y\to x,\ y\in D_\eps} u_\eps(y)$,
$\forall\ x\in \p D_\eps$.

\vskip10pt

\noo{\bf Proof of Theorem \ref{Thm6.1}.}
Similar to Theorem \ref{Thm5.1},
it suffices to prove the strict convexity of $\M$. 
If $\M$ is not strictly convex, it contains a line segment $\ell$.
Choosing the coordinates properly, we assume that $\ell$ is
contained in the hyperplane $\L=\{x_{n+1}=0\}$ and $-e_{n+1}$ is a
normal of $\M$ on $\ell$.  As in the proof
of  Theorem \ref{Thm5.1},  it suffices to rule out three different cases.
\bei
\item [\bf c1)]  Both endpoints of $\ell$ are boundary points of $\M$.
\item [\bf c2)] Only one endpoint of $\ell$ is a boundary point of $\M$.
\item [\bf c3)] None of the endpoints of $\ell$ is a boundary point of $\M$.
\eei
These three cases correspond to Lemmas \ref{L5.1}-\ref{L5.3}, respectively. 

\vskip10pt

\noo {\bf Case c1)}. By a linear transform of the coordinates, we may assume that
\beq  \label{a6.5}
\ell=\{x\ |\ x'=0, -1\le x_n\le 1, x_{n+1}=0\},
\eeq
the endpoints $\pm e_n\in\p\M$, and that $\{x_{n+1}=0\}$ is a supporting plane of $\M$ at $0$, 
where $x'=(x_1,\cdots, x_{n-1})$, $e_n=(0,\cdots, 0, 1,0)\in \R^{n+1}$.
Let $\M_\eps$ be as above.

 \bei
\item [\bf c1.1)] If $\M_\eps$ is a graph of the form
\beq  \label{a6.6}
x_{n+1}=u(x),
\eeq
the proof of Lemma \ref{L5.1} applies to the current case, no change is
needed.

\vskip7pt

\item [\bf c1.2)] If $\M_\eps$ is not a graph of the form \eqref{a6.6}, 
we fix an $\eps\in (0, \eps^*)$, say $\eps=\frac 12 \eps^*$, and
normalize $\M_\eps$ as above. 
We introduce a convex function $\phi=\phi_h$ as follows,
where $h>0$ is a small positive constant.
\beq  \label{a6.8}
\phi(x)=\sup\{a(x)\ |\ 
 a\ \text{is a linear function,} \ a(0)=-h, \  a (x)\le u(x)\ \forall\ x\in D_\eps\}, 
 \eeq
The graph of $\phi$ is a convex cone with vertex at $(0, -h)$.
Extend $\phi$ to $\R^n$ by letting $\phi(tx)+h=t(h+\phi(x)$, 
such that the graph of $\phi$ is a convex cone in $\R^{n+1}$ without boundary.

\vskip7pt

Let $\Gamma = \Gamma_h $ be given by
$$\Gamma = \{x\in D_\eps\ |\ \phi(x)=u_\eps(x)\} , $$ 
and let $\Om$ be the open set bounded by $\Gamma$.
Since $\phi$ is a convex cone, it satisfies \eqref{phib};
in fact, $N_\phi(\R^n)=N_\phi(\{0\})$.
We have the following properties.
\vskip5pt

\bei
\item [(p1)]
For a.e. $h>0$, the measure  $|\Gamma_h|=0$.
\item [] Indeed, if $x\in\Gamma_h\cap\Gamma_{h'}$ for $h\ne h'$,
then $u$ is not twice-differentiable at $x$.
\vskip3pt

\item [(p2)]   $\pom\subset \Gamma$, and $\Om$ is star-shaped.
\vskip3pt

\item [(p3)]   $\pm e_n\in\Gamma$, where $\pm e_n$ are the endpoints of the segment $\ell$ given in \eqref{a6.5}.
\vskip3pt

\item [(p4)]    $u_\eps>\phi$ in $\Om$, $u_\eps=\phi$ on $\p\Om$, and $u_\eps\ge \phi$ in $D_\eps-\bom$.

\item [(p5)] By the local uniform convexity of $\M_0$, 
the graph of $u_\eps$ in $\Om$ is contained in $\M$ (when $h>0$ is small),
hence it is affine maximal.

\eei
\vskip5pt

\noo Consequently, 

\vskip5pt

\bei
\item [(i)] By (p4),  we infer that  $N_\phi(\bom)\subset N_{u_\eps}(\bom)$,
    hence $\ol S[\phi, \Om]\subset\ol S[u_\eps, \Om]$.
\item [(ii)] By (p4), we also have $u_\eps\in \ol S[\phi, \Om]$, when $u_\eps$ is restricted to $\Om$.
\item [(iii)] By the definition of affine maximizers,  $u_\eps$ is the maximizer of $\sup_{\ol S[u_\eps, \Om]} A$.
Hence by (i) (ii), $u_\eps$ is a maximizer of $\sup_{\ol S[\phi, \Om]} A$,   namely, 
\beq  \label{a6.7}
A(u_\eps)=\sup\{A(w)\ |\ w\in \ol S[\Om, \phi]\}. 
\eeq
\eei

We are in the graph case as in \S 5 (Lemma \ref{L5.1}).
Hence we can use the approximation in Theorem \ref{Thm4.1} to reach a
contradiction. By property (p3),  the smoothness of $\p\M_0$ implies
that \eqref{a5.3} holds, which in turn implies \eqref{a5.4} and \eqref{a5.6}. Hence the
proof of Lemma \ref{L5.1} gives a contradiction.
 
 \eei
 
\vskip10pt

\noo {\bf Case c2)}. 
In this case,  by the local uniform convexity of $\M_0$, 
we can make a translation of the coordinates and a rotation of the coordinates such that
 
\bei
\item [(i)]  $\p\M$ is tangent to $\{x'=0\}$ at the endpoint on the boundary.
\vskip3pt

\item [(ii)] $\ell=\{x\ |\ x'=0, 0\le x_n\le b, x_{n+1}=\alpha x_n\},$ where $x'=(x_1, \cdots, x_{n-1})$.
\vskip3pt

\item [(iii)] $-e_{n+1}+\beta e_n$ is a normal of $\M_0$ at $0$.
\eei

\noo where $\alpha, \beta$ are constants. 
There is no loss of generality in assuming that $b=1$. 
Subtracting a linear function such that $\alpha=0$.
Therefore

\bei
\item [(ii')] the origin $0$ is an endpoint of $\ell$ on $\p\M$, $e_n$ is the other endpoint of $\ell$,

\vskip3pt

\item [(iv)]  $\M_\eps\subset\{x_{n+1}\ge 0\}$,

\vskip5pt

\item [(v)] locally near $0$, $\M_\eps$ is the graph of a convex function $u_\eps$ (by property (a2) after Lemma \ref{L6.1}).

\eei

\noo 
Note that in property (v), $\M_\eps$ can be represented as a graph of a convex function $u_\eps$ locally near $0$.
By the convexity of $\M_\eps$,  $u_\eps$ is defined at least in the domain
$$\{x\in\R^n\ |\ {\Small\text{$\frac{|x'|}{\delta_0}$}}+x_n<1, \ x_n>0\},$$
for some small $\delta_0>0$.

Checking the proof of Lemma \ref{L5.2},  one sees that the  proofs of the claims 1-4 in the
proof of Lemma \ref{L5.2} are all valid.
In particular \eqref{a5.19} and \eqref{a5.22} are
still true. Therefore we can make the same blow-up argument as in
\S5 and reach a contradiction.
 
\vskip5pt
 
\noo{\bf Case c3)}. 
We first show that  $\M$ 
is locally uniformly convex and smooth near the boundary $\p\M$. 
Indeed, if this is not true, then there is a sequence of line segments
$\ell_k\subset\M$, such that one of its endpoints converges to the boundary $\p\M$. 
By Cases c1) and c2) above,  no endpoint can be a boundary point. 
Note that the length of $\ell_k$ cannot be too small, for otherwise 
by property (a2) stated before Lemma \ref{L6.2},  $\p\M$ is locally a graph and 
the argument in \S5 implies a contradiction.
Therefore the sequence $\ell_k$ converges to a line segment $\ell\subset\M$
of which at least one point is a boundary point and we are in Case c1) or c2),
which were ruled out above.

For any point $p\in\M$ and any supporting plane $\L_p$ of $\M$ at $p$, 
denote  $T_p=\M\cap\L_p$ (since $\M$ is locally convex,
the set $\M\cap\L_p$ may contain more than one component,
if this is the case we choose  $T_p$ the component which contains the point $p$). 
If $\M$ is not strictly convex near $p$, then 
$T_{p}$ contains a line segment $\ell$.

As above we choose the coordinates such that 
$\L_p=\{x_{n+1}=0\}$,  locally $\M\subset\{x_{n+1}\ge 0\}$.
Let $\M_\eps$ be as above. Since $\M$ is locally strictly convex near $\p\M$,
$\M_\eps$ does not contain any boundary point of $\M$ if $\eps>0$ is small.

If $\M_\eps$ is a graph, namely if it can be represented as \eqref{a6.6}, 
the proof in \S5 (Lemma \ref{L5.3}) applies and we reach a contradiction.

If $\M_\eps$ cannot be represented as a graph as \eqref{a6.6}.
Let $D_\eps$ and $u_\eps$ be as above. 
Let $\phi=\phi_h$ and $\Om$ be as in Case c1), given in \eqref{a6.8}, 
which satisfy the properties (p1) - (p5) after \eqref{a6.8}. 
Then $u_\eps$ is the maximizer of the
affine area functional $A$ in $\ol S[\Om, \phi]$, as in \eqref{a6.7}. 

We are in a similar situation as in Lemma \ref{L5.3}. 
To reach a contradiction, we argue as follows.
Let $u^*_\eps$ be the Legendre transform of $u_\eps$. 
Then by approximation, $u_\eps^*$ is the limit of a sequence of smooth solutions $u^*_m$, 
and hence by \eqref{a4.15}, $u_\eps^*$  is defined in $\Om^*:=N_{\phi}(\Om)$.
Since $\phi$ is a convex cone, $N_\phi$ is a convex domain.
By  \eqref{a4.30}, 
\beq  \label{a6.10}
u^*_\eps>0\ \ \text{on} \ \ \pom^*. 
\eeq 
Hence by Lemma \ref{L2.3} we obtain the upper bound 
\beq  \label{a6.12}
\det D^2 u_\eps^*\le c_2.
\eeq

On the other hand, 
there is no loss of generality in assuming that $\M_\eps$ is normalized.
In Lemma \ref{L2.4} 
we proved that the Gauss curvature of $\M_\eps$ is uniformly bounded near $T_{p}$. 
It implies that there exists a constant $c_1>0$ such that
\beq  \label{a6.11}
\det D^2 u_\eps^*\ge c_1 . 
\eeq
Applying Lemma \ref{L2.2} to $u_\eps^*$, 
we see that  $u_\eps^*$ is strictly convex and $C^1$ smooth, which 
in turn implies that $\M$ is strictly convex, and we reach a contradiction. \hfill $\square$


 \end{document}